\documentclass[hidelinks,onefignum,onetabnum]{siamart251104}

\usepackage{amsfonts}
\usepackage{graphicx}
\usepackage{epstopdf}
\usepackage{algorithmic}
\ifpdf
  \DeclareGraphicsExtensions{.eps,.pdf,.png,.jpg}
\else
  \DeclareGraphicsExtensions{.eps}
\fi

\newsiamremark{remark}{Remark}
\newsiamremark{hypothesis}{Hypothesis}
\crefname{hypothesis}{Hypothesis}{Hypotheses}
\newsiamthm{claim}{Claim}
\newsiamremark{fact}{Fact}
\crefname{fact}{Fact}{Facts}

\headers{Resonance analysis of one-dimensional acoustic media}{Y. Huang, B. Li, P. Liu, and Y. Shao}

\title{Resonance analysis of one-dimensional acoustic media: a propagation matrix approach\thanks{Submitted to the editors DATE.
\funding{This work was partially supported by the Fundamental Research Funds for the Central Universities grant number 226-2025-00192 and the National Key R\&D Program of China grant number 2024YFA1016000.             }}}

\author{Yi Huang\thanks{School of Mathematical Sciences, Zhejiang University, Hangzhou, 310027, China (\email{yi720888@zju.edu.cn, pingliu@zju.edu.cn, shaoyingjie323@zju.edu.cn}).}
\and Bowen Li\thanks{Department of Mathematics, City University of Hong Kong, Kowloon Tong, Hong Kong SAR
  (\email{bowen.li@cityu.edu.hk}).} \and Ping Liu\footnotemark[2] \thanks{ZJU Center for Interdisciplinary Applied Mathematics, Zhejiang University, Hangzhou, 310027, China}
\and Yingjie Shao\footnotemark[2]}

\usepackage{amsopn}

\usepackage{tikz}
\usetikzlibrary{shapes.geometric, shapes.misc, positioning, arrows, decorations.pathreplacing, angles, quotes,calc}

\usepackage{mathtools,bm,amssymb}
\usepackage{adjustbox}
\usepackage{subcaption}
\newtheorem*{example}{Example}

\def\q {\quad}

\def \l{\langle}
\def \r{\rangle}

\def\bb{\begin{equation}
  \left\{\ 
   \begin{aligned} }
\def\ee{   \end{aligned}
  \right.
  \end{equation}}

\def\mm{ \left[
 \begin{matrix}}
\def\nn{\end{matrix} \right] } 

\def\p{\partial}
\def \dd{\mathrm{d}}

\def\d{\delta}

\def\g{\gamma}
\def\s{\sigma}

\def \lad{\lambda}

\def \Lad{\Lambda}

\def\re{\mathfrak{Re}}
\def\im{\mathfrak{Im}}

\def \sdiv {\sdiv_S}

\def \R{\mathbb{R}}

\def \C{\mathbb{C}}

\def \Z{\mathbb{Z}}

\def \e{\mathrm{e}}
\def \i{\mathrm{i}}

\def \ww{\omega}

\def \abs#1{\mathopen| #1 \mathclose|}

\begin{document}

\maketitle


\begin{abstract}
This work analyzes the scattering resonances of general acoustic media in a one-dimensional setting using the propagation matrix approach. Specifically, we characterize the resonant frequencies as the zeros of an explicit trigonometric polynomial. Leveraging Nevanlinna's value distribution theory, we establish the distribution properties of the resonances and demonstrate that their imaginary parts are uniformly bounded, which contrasts with the three-dimensional case. In two classes of high-contrast regimes, we derive the asymptotics of both subwavelength and non-subwavelength resonances with respect to the contrast parameter. Furthermore, by applying the Newton polygon method, we recover the discrete capacitance matrix approximation for subwavelength Minnaert resonances in both Hermitian and non-Hermitian cases, thereby establishing its connection to the propagation matrix framework.
\end{abstract}





\begin{keywords}
Scattering Resonance, Propagation Matrix, Resonance-free Region, Asymptotic Expansion,  Capacitance Matrix, Newton Polygon Method
\end{keywords}

\begin{MSCcodes}
34L20, 34E05, 15A18, 11L03, 30D35
\end{MSCcodes}



\section{Introduction} A fundamental principle in wave physics asserts that, for an object or structure to interact strongly with a wave, such as inducing significant scattering or refraction, its size must be comparable to the wavelength. This principle underpins phenomena like the Abbe diffraction limit, which defines the resolution limit of optical systems, and guides the design of radio antennas. A major scientific challenge is, therefore, to enable wave manipulation at scales far smaller than the wavelength. This has sparked significant interest in the phenomenon of subwavelength resonance, where resonators exhibit strong interactions with incident waves whose wavelengths are orders of magnitude larger than the resonator itself.

A key mechanism for achieving subwavelength resonance is the use of high-contrast media, which consist of structures formed by embedding bounded inclusions with properties that differ significantly from those of the surrounding medium. The stark contrast between the inclusions and the background is a critical prerequisite for the emergence of subwavelength resonance \cite{ammari2018minnaert,meklachi2018asymptotic}. A classic example of this phenomenon is the Minnaert resonance, observed in air bubbles immersed in water \cite{minnaert1933xvi}. Similar subwavelength resonances also appear in other high-contrast systems, including dielectric particles \cite{ammari2023mathematical}, plasmonic particles \cite{ammari2017mathematical}, and Helmholtz resonators \cite{ammari2015mathematical}. The excitation of these resonances has enabled a wide range of innovative wave-based applications, such as superfocusing \cite{ammari2017sub,lanoy2015subwavelength}, cloaking \cite{ammari2013spectral,kohn2014variational}, and wave guiding \cite{ammari2020topologically,ammari2024fano}.


In this work, we will investigate the scattering resonances of acoustic waves in one-dimensional media, where the resonators and the background consist of a finite chain of segments with arbitrary lengths, inter-distances, and material properties (see Figure \ref{fig:setting}). In the high-contrast regime, it has been proven for two- and three-dimensional cases, using layer potential techniques, that resonances exist in the subwavelength regime, giving rise to the so-called Minnaert resonance \cite{ammari2018minnaert,mantile2022origin}. Furthermore, it has been shown that the leading-order terms in their asymptotics with respect to the contrast are characterized by the eigenvalues of the capacitance matrix \cite{ammari2024functional}. However, the boundary integral equation approach is not applicable to one-dimensional systems. Recently, Feppon et al. \cite{Subwavelength_1D_High-Contrast} provided a rigorous analysis of subwavelength resonances in one-dimensional acoustic wave scattering problems, utilizing the variational framework developed in \cite{feppon2024subwavelength}. On the other hand, the propagation matrix approach, which is particularly well-suited for one-dimensional Helmholtz equations (ODE), has been widely employed to analyze resonant wave propagation in topological or disordered media; see \cite{lin2022mathematical,ammari2024subwavelength,thiang2023bulk,ammari2025uniform} for example. 

This raises a natural question about establishing the connection between the propagation matrix approach and the discrete capacitance matrix approach for subwavelength resonances, as well as understanding how the propagation matrix approach operates beyond the subwavelength regime. In this work, we address this gap by analyzing resonances in both the subwavelength and non-subwavelength regimes and demonstrating how the capacitance matrix approximation can be recovered within our framework.





\subsection{Main results}  We consider the one-dimensional acoustic resonance problem \eqref{equ: scattering problem} with the material parameters given by \eqref{equ: k(x) def} and \eqref{eq:contrast}. In this work, we fix the wave-speed ratio $r$ and examine the dependence of the resonances on the complex density ratio $\delta$. Our main contributions are summarized as follows. 

First, using M\"obius transformations, we derive an equivalent analytical condition \eqref{eq:Matrix rep4} for the resonant frequencies in terms of the propagation matrix. In Theorem \ref{thm: resonant frequency property}, we show that the system exhibits identical resonant frequencies for the density ratios $\delta$ and $r^2 / \delta$, and that when $\delta \in \mathbb{R}$, all resonant frequencies are symmetric with respect to the imaginary axis. In particular, when $\delta > 0$, all resonant frequencies lie in the lower half of the complex plane.

Second, using the analytical formulation \eqref{eq:Matrix rep4}, we explicitly compute the analytic function $f(k; \sigma)$, with $\sigma = \d/r$, whose zeros characterize all the resonant frequencies (Theorem \ref{thm: resonant frequency f(z;mu) zeros}). Then, we establish the general distribution pattern of these zeros and characterize the resonance-free region in Theorem \ref{thm: f(z;mu) zeros}, by applying Nevanlinna value distribution theory. Moreover, we prove that as $\delta \to 0$, the zeros of $f(\cdot; \sigma)$, and hence the resonant frequencies, converge to the set $E = \cup_{j=1}^{2N-1} (\pi \mathbb{Z}/t_j)$. The number of zeros near each $k_0 \in E$ is determined by the multiplicity $n(k_0)$ of the corresponding zero of a limiting analytic function $f(\cdot; 0)$; see Theorem \ref{thm: zero near k}. 
For the case of simple zeros ($n(k_0) = 1$), we employ the implicit function theorem to derive the first-order asymptotics of the corresponding resonances as $\delta \to 0$ and $\delta \to \infty$ in Theorem \ref{thm:resonancesexpan1}. 

Third, we further investigate the asymptotic behavior of subwavelength resonances as $\delta \to 0$ and $\delta \to \infty$ in Theorem \ref{thm: subwavelength_resonant_frequencies} and Theorem \ref{thm:nonrecip} for Hermitian and non-Hermitian systems, respectively. Unlike the expansion in the previous section, where the implicit function theorem was applicable, the characteristic function $f(k;\sigma)$ in \eqref{eq: f(z; sigma) expansion} possesses a high-order root at zero, necessitating the use of the Newton polygon method (see Appendix \ref{app: Newton Polygon Method}) from multivariate complex analysis. This enables us to recover the capacitance matrix theory \cite{Subwavelength_1D_High-Contrast} from a novel complex analytic perspective. Moreover, for subwavelength resonant modes, we find that as $\delta \to 0$, the eigenmodes are approximately constant within the resonators and nearly linear within the spacing layers, with amplitudes governed by the capacitance matrix eigenvectors. Conversely, as $\delta \to \infty$, the eigenmodes are approximately linear within the resonators and nearly constant within the spacings. Another interesting finding is that subwavelength resonances persist in one-dimensional systems as $\delta\to\infty$, yet are entirely absent in three-dimensional structures (Section \ref{sec:threeddeltainfinity1}).




\subsection{Outlines} The paper is organized as follows. In Section \ref{sec:preli}, we characterize the resonances in one-dimensional acoustic media using the propagation matrix method. In Section \ref{sec:characresonantfreq}, we establish the general distribution properties of the resonances and derive the asymptotics for two high-contrast regimes. Then, in Section \ref{sec:capacitancematrixtheroy1}, we present how the capacitance matrix approximation for subwavelength resonances can be recovered within the propagation matrix framework. Finally, in Section \ref{sec:nonresys}, we generalize the results from Section \ref{sec:capacitancematrixtheroy1} to non-reciprocal systems.

\section{Preliminaries} \label{sec:preli}
 
This section introduces the scattering resonance problem for acoustic waves in one-dimensional inhomogeneous media. We employ the propagation matrix method to characterize the resonant frequencies. The main result, Theorem \ref{thm: resonant frequency property}, establishes fundamental properties of scattering resonances and unveils a duality between two distinct contrast regimes.




\subsection{Model setting}\label{sec:modelsethermitian}
We consider a one-dimensional chain of $N$ disjoint, identical resonators $D_j \coloneqq (x_j^{-}, x_j^{+})$, where $(x_j^{\pm})_{1 \leq j \leq N} \subset \R$ are the $2N$ boundary points satisfying $x_j^{-} < x_j^{+} < x_{j+1}^{-}$ for all $1 \leq j \leq N-1$. The length of each resonator is denoted by $\ell_j = x_j^{+} - x_j^{-}$, and the spacing between the $j$-th and $(j+1)$-th inclusions is given by $s_j = x_{j+1}^{-} - x_j^{+}$. The configuration of the system is illustrated in Fig.\,\ref{fig:setting}. 

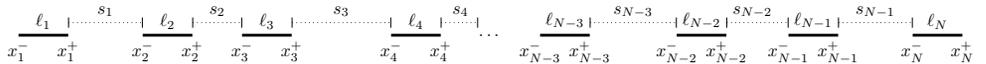
\begin{figure}[htbp!]
    \centering
    \begin{adjustbox}{width=\textwidth}
    \begin{tikzpicture}
        \coordinate (x1l) at (1,0);
        \path (x1l) +(1,0) coordinate (x1r);
        \path (x1r) +(0.75,0.7) coordinate (s1);
        \path (x1r) +(1.5,0) coordinate (x2l);
        \path (x2l) +(1,0) coordinate (x2r);
        \path (x2r) +(0.5,0.7) coordinate (s2);
        \path (x2r) +(1,0) coordinate (x3l);
        \path (x3l) +(1,0) coordinate (x3r);
        \path (x3r) +(1,0.7) coordinate (s3);
        \path (x3r) +(2,0) coordinate (x4l);
        \path (x4l) +(1,0) coordinate (x4r);
        \path (x4r) +(0.4,0.7) coordinate (s4);
        \path (x4r) +(1,0) coordinate (dots);
        \path (dots) +(1,0) coordinate (x5l);
        \path (x5l) +(1,0) coordinate (x5r);
        \path (x5r) +(1.75,0) coordinate (x6l);
        \path (x5r) +(0.875,0.7) coordinate (s5);
        \path (x6l) +(1,0) coordinate (x6r);
        \path (x6r) +(1.25,0) coordinate (x7l);
        \path (x6r) +(0.525,0.7) coordinate (s6);
        \path (x7l) +(1,0) coordinate (x7r);
        \path (x7r) +(1.5,0) coordinate (x8l);
        \path (x7r) +(0.75,0.7) coordinate (s7);
        \path (x8l) +(1,0) coordinate (x8r);
        \draw[ultra thick] (x1l) -- (x1r);
        \node[anchor=north] (label1) at (x1l) {$x_1^{-}$};
        \node[anchor=north] (label1) at (x1r) {$x_1^{+}$};
        \node[anchor=south] (label1) at ($(x1l)!0.5!(x1r)$) {$\ell_1$};
        \draw[dotted,|-|] ($(x1r)+(0,0.25)$) -- ($(x2l)+(0,0.25)$);
        \draw[ultra thick] (x2l) -- (x2r);
        \node[anchor=north] (label1) at (x2l) {$x_2^{-}$};
        \node[anchor=north] (label1) at (x2r) {$x_2^{+}$};
        \node[anchor=south] (label1) at ($(x2l)!0.5!(x2r)$) {$\ell_2$};
        \draw[dotted,|-|] ($(x2r)+(0,0.25)$) -- ($(x3l)+(0,0.25)$);
        \draw[ultra thick] (x3l) -- (x3r);
        \node[anchor=north] (label1) at (x3l) {$x_3^{-}$};
        \node[anchor=north] (label1) at (x3r) {$x_3^{+}$};
        \node[anchor=south] (label1) at ($(x3l)!0.5!(x3r)$) {$\ell_3$};
        \draw[dotted,|-|] ($(x3r)+(0,0.25)$) -- ($(x4l)+(0,0.25)$);
        \node (dots) at (dots) {\dots};
        \draw[ultra thick] (x4l) -- (x4r);
        \node[anchor=north] (label1) at (x4l) {$x_4^{-}$};
        \node[anchor=north] (label1) at (x4r) {$x_4^{+}$};
        \node[anchor=south] (label1) at ($(x4l)!0.5!(x4r)$) {$\ell_4$};
        \draw[dotted,|-|] ($(x4r)+(0,0.25)$) -- ($(dots)+(-.25,0.25)$);
        \draw[ultra thick] (x5l) -- (x5r);
        \node[anchor=north] (label1) at (x5l) {$x_{N-3}^{-}$};
        \node[anchor=north] (label1) at (x5r) {$x_{N-3}^{+}$};
        \node[anchor=south] (label1) at ($(x5l)!0.5!(x5r)$) {$\ell_{N-3}$};
        \draw[dotted,|-|] ($(x5r)+(0,0.25)$) -- ($(x6l)+(0,0.25)$);
        \draw[ultra thick] (x6l) -- (x6r);
        \node[anchor=north] (label1) at (x6l) {$x_{N-2}^{-}$};
        \node[anchor=north] (label1) at (x6r) {$x_{N-2}^{+}$};
        \node[anchor=south] (label1) at ($(x6l)!0.5!(x6r)$) {$\ell_{N-2}$};
        \draw[dotted,|-|] ($(x6r)+(0,0.25)$) -- ($(x7l)+(0,0.25)$);
        \draw[ultra thick] (x7l) -- (x7r);
        \node[anchor=north] (label1) at (x7l) {$x_{N-1}^{-}$};
        \node[anchor=north] (label1) at (x7r) {$x_{N-1}^{+}$};
        \node[anchor=south] (label1) at ($(x7l)!0.5!(x7r)$) {$\ell_{N-1}$};
        \draw[dotted,|-|] ($(x7r)+(0,0.25)$) -- ($(x8l)+(0,0.25)$);
        \draw[ultra thick] (x8l) -- (x8r);
        \node[anchor=north] (label1) at (x8l) {$x_{N}^{-}$};
        \node[anchor=north] (label1) at (x8r) {$x_{N}^{+}$};
        \node[anchor=south] (label1) at ($(x8l)!0.5!(x8r)$) {$\ell_N$};
        \node[anchor=north] (label1) at (s1) {$s_1$};
        \node[anchor=north] (label1) at (s2) {$s_2$};
        \node[anchor=north] (label1) at (s3) {$s_3$};
        \node[anchor=north] (label1) at (s4) {$s_4$};
        \node[anchor=north] (label1) at (s5) {$s_{N-3}$};
        \node[anchor=north] (label1) at (s6) {$s_{N-2}$};
        \node[anchor=north] (label1) at (s7) {$s_{N-1}$};
    \end{tikzpicture}
    \end{adjustbox}
    \caption{A chain of $N$ resonators, with lengths
    $(\ell_j)_{1\leq j\leq N}$ and spacings $(s_{j})_{1\leq j\leq N-1}$.}
    \label{fig:setting}
\end{figure}
We denote the collection of subwavelength resonators by the set
\begin{align*}
   D\coloneqq \bigcup_{j=1}^N(x_j^{-},x_j^{+})\,.
\end{align*}
In this work, we study the one-dimensional Helmholtz equation for the acoustic wave propagation in a heterogeneous medium associated with $D$:
\begin{align}
    \frac{\omega^{2}}{\kappa(x)}u(x) +\frac{\dd}{\dd x}\left( \frac{1}{\rho(x)}\frac{\dd}{\dd
    x}  u(x)\right) =0,\qquad x \in\R\,,
    \label{eq:helmoltz equation 1}
\end{align}
where the bulk modulus $\kappa(x)$ and the density $\rho(x)$ of the medium are assumed to be piecewise constant inside and outside the resonators:
\begin{align}\label{equ: k(x) def}
    \kappa(x)=
    \begin{dcases}
        \kappa_b, & x\in D,\\
        \kappa, &  x\in\R\setminus D,
    \end{dcases}\quad\text{and}\quad
    \rho(x)=
    \begin{dcases}
        \rho_b, & x\in D,\\
        \rho, &  x\in\R\setminus D.
    \end{dcases}
\end{align}
The wave speeds inside the set $D$ of resonators and in the background $\R \setminus D$ are denoted by $v_b$ and $v$, respectively, with the corresponding wave numbers $k_b$ and $k$. The contrasts between the densities and wave speeds of the resonators and the background medium are denoted by $\delta$ and $r$, respectively. Specifically, let 
\begin{align} \label{eq:contrast}
    v_b:=\sqrt{\frac{\kappa_b}{\rho_b}}, \q v:=\sqrt{\frac{\kappa}{\rho}},\q
    k_b:=\frac{\omega}{v_b},\q k:=\frac{\omega}{v},\q
    \delta:=\frac{\rho_b}{\rho},\q r:=\frac{v}{v_b}.
\end{align}
Due to their physical significance, the parameters defined in \eqref{eq:contrast} are typically treated as positive real numbers. However, as established in Theorem \ref{thm: resonant frequency f(z;mu) zeros}, the resonant frequencies of the system are uniquely determined by the parameters $\delta$ and $r$, which can, in general, be complex. In this work, we fix \(r\) as a positive constant, allow \(\delta \in \mathbb{C}\), and study the dependence of the resonance \(\omega\) on the complex parameter \(\delta\). 



Using the parameters defined in \eqref{eq:contrast}, the Helmholtz equation \eqref{eq:helmoltz equation 1} can be reformulated as follows:
\begin{align}
    \label{equ: scattering problem}
    \begin{dcases}
        \frac{\dd{^2}}{\dd x^2}u(x)+ k^2 u(x) = 0, & x\in \R \setminus D ,\\
        \frac{\dd{^2}}{\dd x^2}u(x)+ k_b^2 u(x) = 0, & x\in D ,\\
        u\vert_{+}(x^{\pm}_{{j}}) = u\vert_{-}(x^{\pm}_{{j}}) , &  1\leq j\leq N ,\\
        \left.\frac{\dd u}{\dd x}\right\vert_{\pm}(x^{\mp}_{{j}}) = \delta\left.\frac{\dd u}{\dd x}\right\vert_{\mp}(x^{\mp}_{{j}}), & 1\leq j\leq N ,\\
        \left(\frac{\dd}{\dd |x|} - \mathrm{i} k \right) u = 0 & \text{for } x \in (-\infty, x^{-}_{1}) \cup (x^{+}_{N}, +\infty),  \\
    \end{dcases}
\end{align}
where for a one-dimensional function $w$, we denote its left and right limits, if they exist, by $w\vert_{\pm}(x) \coloneqq \lim_{s \to 0^+} w(x\pm s)$. 
We say that $\ww \in \C$ is a \emph{resonant frequency (resonance)} if the scattering problem \eqref{equ: scattering problem} admits non-trivial solutions $u$, which are called resonant modes. 

\subsection{Propagation matrix and M\"obius transformation}

For the scattering problem (\ref{equ: scattering problem}), the zero frequency, \(\omega^* = 0\), is always a resonant frequency, as any constant function satisfies the equation. In this work, however, we focus on the non-trivial resonant frequencies (\(\omega \neq 0\)). To characterize these non-trivial resonances, we introduce the propagation matrix method in this section.

Let us begin with the following second-order ODE for $0 \neq k \in \C$:
\begin{align} \label{eq:2ndode}
\frac{\dd^2}{\dd x^2}u(x) + k^2 u(x) = 0, \quad x \in (0, a),
\end{align}
where the solution can be expressed as $u(x) = A\e^{\i kx} + B\e^{-\i kx}$, with $A, B \in \C$ uniquely given by 
\begin{align*}
    A=\frac{u(0)+\frac{1}{\i k}u'(0)}{2}, \q  B=\frac{u(0)-\frac{1}{\i k}u'(0)}{2}.
\end{align*}
It follows that 
\begin{equation} \label{eq:soluprop}
    \begin{aligned}
          u(a)&= A \e^{\i ka}+ B \e^{-\i ka}=\cos(ka)u(0)+\frac{1}{k}\sin(ka)u'(0),\\
    u'(a)&=\i k A \e^{\i ka}-\i k B \e^{-\i ka}=-k\sin(ka)u(0)+\cos(ka)u'(0).
    \end{aligned}
\end{equation}
Define matrices, for $z, k \in \C$ and $a \in \R$,
\begin{equation} \label{eq:deftka}
    T(k,a) := \begin{pmatrix}
        \cos(ka) & \sin(ka)\\
        -\sin(ka) & \cos(ka)
    \end{pmatrix},\q
    M(z) := \begin{pmatrix}
        1 & 0\\
        0 & z
    \end{pmatrix}.
\end{equation}
The propagation matrix $P(k,a)$ for \eqref{eq:2ndode} can be defined by \eqref{eq:soluprop} as follows: 
\begin{align} \label{eq:solpropa}
    \begin{pmatrix}
        u(a)\\u'(a)
    \end{pmatrix}=
   P(k,a)\begin{pmatrix}
        u(0)\\u'(0)
    \end{pmatrix}
    ,\q P(k,a) :=  
    M(k)T(k,a)M\left(\frac{1}{k}\right).
\end{align}

We now turn to the scattering problem \eqref{equ: scattering problem}, and the above formula \eqref{eq:solpropa} implies 
\begin{align}\label{equ: transfer matrix inclusions}
\begin{pmatrix}
u|_-(x_{j+1}^-)\\u'|_-(x_{j+1}^-)
\end{pmatrix} = P(k,s_j)  \begin{pmatrix}
    u|_+(x_j^+)\\u'|_+(x_j^+)
\end{pmatrix},\quad 1\leq j\leq N-1,
\end{align}
and
\begin{align} \label{eq:transferreso}
\begin{pmatrix}
    u|_-(x_j^+)\\u'|_-(x_j^+)
\end{pmatrix} = P(rk,\ell_j) \begin{pmatrix}
    u|_+(x_j^-)\\u'|_+(x_j^-)
\end{pmatrix},
\quad 1\leq j\leq N.    
\end{align}
Note from the transmission boundary condition in \eqref{equ: scattering problem} that for $1\leq j\leq N$, 
\begin{equation} \label{eq:jump}
    \begin{pmatrix}
    u|_+(x_j^+)\\u'|_+(x_j^+)
\end{pmatrix}=M\left(\frac{1}{\delta}\right)\begin{pmatrix}
    u|_-(x_j^+)\\u'|_-(x_j^+)
\end{pmatrix}, \q  
\begin{pmatrix}
    u|_+(x_j^-)\\u'|_+(x_j^-)
\end{pmatrix}=M\left(\delta\right)\begin{pmatrix}
    u|_-(x_j^-)\\u'|_-(x_j^-)
\end{pmatrix}.
\end{equation}
Combining \eqref{eq:transferreso} and \eqref{eq:jump} gives 
\begin{align}\label{equ: transfer matrix resonantors}
\begin{pmatrix}
    u|_+(x_j^+)\\u'|_+(x_j^+)
\end{pmatrix}=M\left(\frac{rk}{\delta}\right)T(rk,\ell_j)M\left(\frac{\delta}{rk}\right)\begin{pmatrix}
    u|_-(x_j^-)\\u'|_-(x_j^-)
\end{pmatrix}\,, \quad 1\leq j\leq N,
\end{align}
thanks to $M(z_1)M(z_2) = M(z_1 z_2)$. By the radiation condition in \eqref{equ: scattering problem}, there holds 
\begin{equation*}
    u(x)=\begin{cases}
        c_1\e^{-\i k(x-x_1^-)},&x<x_1^-,\\
        c_2\e^{\i k(x-x_N^+)},&x>x_N^+,
    \end{cases}
\end{equation*}
with $c_1,c_2 \in \C$. We assume both $c_1$ and $c_2$ are non-zero; otherwise, $u \equiv 0$. Consequently, we have, using $M(z)$ in \eqref{eq:deftka}, 
\begin{align}\label{equ: transfer matrix infty}
\begin{pmatrix}
    u_-(x_1^-)\\u'|_-(x_1^-)
\end{pmatrix}
=c_1M(k)\begin{pmatrix}
    1\\-\i
\end{pmatrix},\quad
\begin{pmatrix}
    u_+(x_N^+)\\u'|_+(x_N^+)
\end{pmatrix}
=c_2M(k)\begin{pmatrix}
    1\\\i
\end{pmatrix}.
\end{align}
Thus, for a non-trivial resonance $\ww = kv \neq 0$, the equations (\ref{equ: transfer matrix inclusions}), (\ref{equ: transfer matrix resonantors}) and (\ref{equ: transfer matrix infty}) imply
\begin{align} \label{eq:Matrix rep2}
    c\begin{pmatrix}
        1 \\ \i
    \end{pmatrix}=M\left(\frac{r}{\d}\right)T(rk,l_N)M\left(\frac{\d}{r}\right)T(k,s_{N-1})\cdots M\left(\frac{r}{\d}\right)T(kr,l_1)M\left(\frac{\d}{r}\right)
    \begin{pmatrix}
        1 \\-\i
    \end{pmatrix},
\end{align}
where $c=c_2/c_1\ne0$. We summarize the above discussion as follows. 
\begin{lemma}
    $\ww \in \C \backslash \{0\}$ is a resonance for \eqref{equ: scattering problem} if and only if the corresponding wave number $k = \ww / v$ satisfies \eqref{eq:Matrix rep2} for some $c \neq 0$. In this case, $c$ is uniquely determined by $k$. 
\end{lemma}


We introduce a vector 
\begin{equation} \label{def:vectort}
    \bm t:=(r\ell_1,\ s_{1},\ r\ell_{2},\ s_{2},\ \cdots,\ r\ell_{N-1},\ s_{N-1},\ r\ell_N)^\top\in\R^{2N-1}_{>0},
\end{equation}
and denote
\[
\sigma=\frac{\d}{r}.
\]
Using $T(rk, a) = T(k, ra)$ by definition \eqref{eq:deftka}, the equation \eqref{eq:Matrix rep2} can be written as 
\begin{align} \label{eq:Matrix rep3}
    c\begin{pmatrix}
        1 \\\i
    \end{pmatrix}=M\left(\frac{1}{\sigma}\right)T(k,t_{2N-1})M(\sigma)T(k,t_{2N-2})\cdots M\left(\frac{1}{\sigma}\right)T(k,t_{1})M(\s)
    \begin{pmatrix}
        1 \\-\i
    \end{pmatrix},
\end{align}
where $t_j$ is the $j-$th component of $\bm t$.
Letting $e_{\pm} = (1, \pm \i)^\top$, we have, by \eqref{eq:deftka}, 
\begin{align*}
    &T(k,a)e_{\pm}=\e^{\pm \i ka}e_{\pm}, \q M(\sigma)e_{\pm}=\frac{1+\sigma}{2}e_{\pm} +\frac{1-\s}{2}e_{\mp}\,.
\end{align*}
This enables us to further rewrite \eqref{eq:Matrix rep3} as: for $c \neq 0$,
\begin{align} \label{eq:Matrix rep4}
    c\begin{pmatrix}
        1 \\0
    \end{pmatrix}=R\left(\frac{1}{\s}\right)L(t_{2N-1}k)R(\s)L(t_{2N-2}k)\cdots R\left(\frac{1}{\s}\right)L(t_1k)R(\s)
    \begin{pmatrix}
        0 \\1
    \end{pmatrix}, 
\end{align}
where for $z \in \C$, 
\begin{align} \label{def:rlmatrix}
    R(z) := \begin{pmatrix}
        \frac{1+ z}{2} & \frac{1-z}{2}\\
        \frac{1- z}{2} & \frac{1+z}{2}
    \end{pmatrix},
    & \quad
    L(z) := \begin{pmatrix}
         \e^{\i z} & 0\\
        0 & \e^{-\i z}
    \end{pmatrix}.
\end{align}

The characterization in \eqref{eq:Matrix rep4} is fundamental to all the main results of this paper. As a first application, we next establish some general properties of the resonant frequencies, with a more detailed treatment reserved for Sections \ref{sec:characresonantfreq} and \ref{sec:capacitancematrixtheroy1}.

We define the M\"obius transformation (see \cite[Chapter 3]{wegert2012visual}) associated with a complex matrix $A=(a_{ij})_{2\times2}\in\C^{2\times2}$ by
\begin{align} \label{def:moboiustran}
    f_A(z):=\frac{a_{11}z+a_{12}}{a_{21}z+a_{22}} ,\quad z \in \C\cup \{\infty\},
\end{align}
and for $\Phi = (\phi_1, \phi_2)^\top \in \C^2$, we define $\mathcal{M}(\Phi) := \phi_1 / \phi_2 \in \C \cup \{\infty\}$. It is straightforward to verify that 
\begin{align} \label{propmoub1}
    \mathcal{M}(A\Phi)=f_A \circ \mathcal{M}(\Phi),\quad f_{AB}=f_A \circ f_B.
\end{align}
We then introduce two classes of rational functions by \eqref{def:moboiustran}, based on matrices $R(\sigma)$ and $L(t_j k)$, 
\begin{align} \label{propmoub2}
    &f_{\s}(z):=f_{R(\s)}(z)=\frac{(1+\s)z+(1-\s)}{(1-\s)z+(1+\s)},\quad g_{j}(z):=f_{L(t_jk)}(z)=\e^{2\i kt_j}z.
\end{align}
Applying $\mathcal{M}$ on both sides of  \eqref{eq:Matrix rep4} and using \eqref{propmoub1}-\eqref{propmoub2}, we obtain
\[
   f_{\s^{-1}}\circ g_{2N-1} \circ f_\s \circ \cdots \circ g_{1} \circ f_\s\circ \mathcal M \left(
\begin{pmatrix}0\\
1
\end{pmatrix}
   \right)= \mathcal M \left(
\begin{pmatrix}1\\
0
\end{pmatrix}\right),
\]
that is, 
\begin{align} \label{eq:Mobius rep1}
    f_{\s^{-1}}\circ g_{2N-1} \circ f_\s \circ \cdots \circ g_{1} \circ f_\s(0)=\infty, 
\end{align}
which is equivalent to \eqref{eq:Matrix rep4}. We now conclude this section with the following theorem.




\begin{theorem}\label{thm: resonant frequency property}
For resonant frequencies $\ww$ of the problem (\ref{equ: scattering problem}), we have 
\begin{enumerate}
     \item If $\d \in \R$, all resonances are symmetric with respect to the imaginary axis.
    \item If $\d>0$, all non-trivial resonant frequencies have negative imaginary parts.
    \item Let $0\ne\delta_0\in\C$, then (\ref{equ: scattering problem}) with $\delta=\delta_0$ and $\delta=\tfrac{r^2}{\delta_0}$ has the same resonances $\omega = kv$. Moreover, let $u(x)$ and $v(x)$ be the resonant modes associated with a resonance $\ww$ for $\delta=\delta_0$ and $\delta= \tfrac{r^2}{\delta_0}$, respectively, satisfying 
 $u(x)=\e^{-\i\frac{\ww}{v}x}$,  $v(x)=t\e^{-\i\frac{\ww}{v}x}$, $x< x_1^{-}$ for some $0\ne t\in\C$. Then, \begin{equation}\label{equ:eigenmodesrelation1}
    \begin{aligned}
         &u'(x) = -\mathrm{i} \frac{\d}{tr}k_b  v(x), \quad v'(x) = -\mathrm{i} \frac{tr}{\d}k_b  u(x), \quad x\in D\\
         &u'(x) = -\mathrm{i}t^{-1}  k  v(x), \quad v'(x) = -\mathrm{i}t k  u(x), \quad x\in \R\setminus D,
     \end{aligned}
     \end{equation}
where, at each endpoint $x_j^\pm$, $u'(x)$ is interpreted as either \(\left. \tfrac{\dd u}{\dd x}\right|_{-}\) or \(\left.\tfrac{\dd u}{\dd x}\right|_{+}\).
\end{enumerate}
\end{theorem}
\begin{remark}
Statement (3) asserts that, in one-dimensional space, a system with high-density resonators embedded in a low-density background shares exactly the same set of resonant frequencies as the system with the density contrast reversed. However, this property is unique to one dimension and does not generally hold in higher-dimensional spaces. Specifically, we will demonstrate its failure in the three-dimensional case in Section \ref{sec:threeddeltainfinity1}.
\end{remark}


\begin{proof}[Proof of Theorem \ref{thm: resonant frequency property}]
    For (1), we only need to take the conjugate of both sides of (\ref{eq:Mobius rep1}) which shows $(\d,-\overline{k})$ also satisfies \eqref{eq:Mobius rep1}. 
    
    For (2), it relies on the fact that $|f_\s(z)|<1$ for $|z|<1$, which can be proved by applying the following two properties: $|z|<1$ if and only if $\re(\frac{z-1}{z+1})<0$; and $\tfrac{f_\s(z)-1}{f_\s(z)+1}=\s\frac{z-1}{z+1}$. Moreover, if  $\im(k)\geq 0$, noting $r>0$, then $|g_j(z)|\leq |z|$ for any $ z\in\C$. Therefore, the modulus of the left-hand side of (\ref{eq:Mobius rep1}) must always remain less than \(1\) and can never approach \(\infty\). It follows that $\im(k) < 0$.  

For the first part of (3), the proof relies on the key observation \(f_{\sigma}(-z) = -f_{\sigma^{-1}}(z)\). Then, for \((\delta, k)\) satisfying \eqref{eq:Mobius rep1}, we have
    \begin{align*}
        &f_{\s}\circ g_{2N-1} \circ f_{\s^{-1}} \circ \cdots \circ g_{1} \circ f_{\s^{-1}}(0)\\
        =&f_{\s}\circ g_{2N-1} \circ f_{\s^{-1}} \circ \cdots \circ g_1 \circ -f_{\s}(0)\\
        =&f_{\s}\circ g_{2N-1} \circ f_{\s^{-1}} \circ \cdots \circ (-g_{1}) \circ f_{\s}(0)
        \\
        =&\cdots\\=&-f_{\s^{-1}}\circ g_{2N-1} \circ f_{\s} \circ \cdots \circ g_{1} \circ f_{\s}(0) = \infty.
    \end{align*}
Thus, \((\tfrac{r^2}{\delta}, k)\) also satisfies \eqref{eq:Mobius rep1}, completing the proof.
    
    For the second part of (3), we first show the following claim:

\smallskip 


    \noindent
    \textbf{Claim:} Assume \(0 \neq k \in \mathbb{C}\), and the functions \(u\) and \(v\) satisfy the Helmholtz equation \(w'' + k^2 w = 0\) on the interval \([0, a]\). If the boundary conditions $u'(0) = \mathrm{i} k t v(0)$ and $v'(0) = \mathrm{i} k t^{-1} u(0)$ hold at the endpoint \(0\) for some \(t \in \mathbb{C}\), then for all \(x \in [0, a]\), we have $u'(x) = \mathrm{i} k t v(x)$ and $v'(x) = \mathrm{i} k t^{-1} u(x)$. 
\smallskip 

\noindent
The claim is derived by using the propagation matrix in (\ref{eq:solpropa}). Specifically, we have
    \begin{align*}
        u(x)&=\cos(kx)u(0)+\frac{1}{k}\sin(kx)u'(0)=\cos(kx)\frac{t}{\i k}v'(0)+\frac{1}{k}\sin(kx)\i ktv(0)\\
        &=\frac{t}{\i k}\left(\cos(kx)v'(0)-k\sin(kx)v(0)\right) = \frac{t}{\i k}v'(x). 
    \end{align*}
Similarly, using the same approach, we can derive that \(u'(x) = \mathrm{i} k t v(x)\). 
According to the above claim, to prove the relations \eqref{equ:eigenmodesrelation1} between the eigenmodes, it suffices to show that \eqref{equ:eigenmodesrelation1} holds at every endpoint $x_j^\pm$. By the assumption on \(u(x)\) and \(v(x)\), \eqref{equ:eigenmodesrelation1} is satisfied on $(-\infty,x_1^-)$. Then, using the jump relation, we have 
    \begin{align*}
       &\left.\frac{\dd u}{\dd x}\right\vert_{+}(x_1^{-})=\d\left.\frac{\dd u}{\dd x}\right\vert_{-}(x_1^{-})=-\mathrm{i} \frac{\d}{tr}k_bv(x), \\
       &\left.\frac{\dd v}{\dd x}\right\vert_{+}(x_1^{-})=\frac{r^2}{\d}\left.\frac{\dd v}{\dd x}\right\vert_{-}(x_1^{-})=-\mathrm{i} \frac{tr}{\d}k_bu(x),
    \end{align*}
    which implies \eqref{equ:eigenmodesrelation1} on $(x_1^-,x_1^+)$. 
   Continuing this process for all intervals, we complete the proof.
\end{proof}

\section{Characterization of resonant frequencies} \label{sec:characresonantfreq}

In this section, we first reformulate the acoustic resonance problem \eqref{equ: scattering problem} (equivalently, the equation \eqref{eq:Matrix rep4}) as the study of zeros of a trigonometric polynomial (Theorem \ref{thm: resonant frequency f(z;mu) zeros}). Leveraging properties of zeros of trigonometric polynomials, we then establish a distribution law for the resonances (Theorems \ref{thm: f(z;mu) zeros} and \ref{thm: zero near k}). Finally, we derive the leading-order asymptotic expansions of the resonances $k$ as $\d \to 0$ and $\d \to \infty$ (Theorem \ref{thm:resonancesexpan1}).

\subsection{Distribution property}

Building on the resonance characterization \eqref{eq:Matrix rep4}, we define the transformed total propagation matrix for the problem \eqref{equ: scattering problem}: 
\begin{align}\label{equ: M_total def}
    M_{tot}(k; \s) := \frac{(4\s)^N}{(1+\s)^{2N}} R\left(\frac{1}{\s}\right)L(t_{2N-1}k)R(\s)L(t_{2N-2}k) \cdots R\left(\frac{1}{\s}\right)L(t_1k)R(\s),
\end{align}
where the factor $\tfrac{(4\s)^N}{(1+\s)^{2N}}$ is a technical scaling introduced for the subsequent asymptotic analysis. 
Then, $\ww$ is a resonant frequency if and only if the associated wavenumber $k = \ww/v$ satisfies  
\begin{align} \label{cond:reson}
    M_{tot}(k;\s)_{1, 2}\neq 0, \quad M_{tot}(k;\s)_{2, 2}= 0\,.
\end{align}
Moreover, note that $M_{tot}(k, \sigma)$ is an invertible matrix for $\sigma \neq 0$, due to the invertibility of $R$ and $L$ from definition \eqref{def:rlmatrix}. It follows directly from \eqref{cond:reson} that a resonance $\ww \neq 0$ is characterized solely by the condition $M_{tot}(k; \s)_{2, 2} = 0$. By directly expanding the matrix multiplication and computing $M_{tot}$, we establish the following theorem, which serves as a foundation for the subsequent discussions.

 
\begin{theorem}\label{thm: resonant frequency f(z;mu) zeros}
If $0\ne\omega=kv$, then $\omega$ is a resonant frequency if and only if $k$ is a zero of a analytic function $f(k;\s)$ in $k$ defined as 
\begin{equation}\label{eq: f(z; sigma) expansion}
    \begin{aligned}
        f(k;\s):&=M_{tot}(k;\s)_{2, 2}\\
        &=\sum_{\bm\alpha\in\{-1,1\}^{2N-1}}{(-1)^{\sum_{j=1}^{2N}j\epsilon_{\alpha_j,\alpha_{j-1}}}\left(\frac{1-\s}{1+\s}\right)^{\sum_{j=1}^{2N}\epsilon_{\alpha_j,\alpha_{j-1}}}\e^{\i\l\bm\alpha,\bm t\r k}}, 
    \end{aligned}
\end{equation}
where
\begin{align*}
    \bm\alpha=(\alpha_1,\cdots,\alpha_{2N-1})\in\{-1,1\}^{2N-1}\,, \quad \alpha_0=\alpha_{2N}=-1\,, \quad \epsilon_{jk}=\begin{cases}
    1,&j\ne k,\\
    0,&j=k.
\end{cases} 
\end{align*}
\end{theorem}

\begin{example}\label{thm: examples of f(z;mu)}
(i) When \(\sigma = 1\) (namely, \(\delta = r > 0\)), \(f(k; \sigma)\) in (\ref{eq: f(z; sigma) expansion}) is given by  
\[
f(k;1)=\e^{-\i||\bm t||_1k}.
\]
In this case, the scattering problem (\ref{equ: scattering problem}) has only the trivial resonance $\omega = 0$, as $e^z \neq 0$ for any $z \in \C$.


(ii) When $N=1$, there is only one resonator. $f(k;\s)$ in (\ref{eq: f(z; sigma) expansion}) is given by 
\[
f(k;\s)=\e^{-\i r\ell_1k}-\left(\frac{1-\s}{1+\s}\right)^2\e^{\i r\ell_1k}.
\]
It follows that if \(\sigma = 1\), there is only the trivial resonance \(\omega^* = 0\), and that in the case of \(\sigma \neq 1\), all resonances are given by
\[
\ww^*=0,\q \omega_n=\frac{v}{r\ell_1}\left(n\pi+\i \ln\left|\frac{1-\s}{1+\s}\right|\right),\ n\in\Z.
\]
This implies that as \(\delta \to 0_+\) or \(\delta \to +\infty\), \(\omega_n \to \tfrac{vn  \pi}{r\ell_1}\). Hence, only \(\omega^*\) and \(\omega_0\) are subwavelength resonant frequencies. 
When \(\delta \to r\), all non-trivial resonant frequencies tend to \(\infty\) (since the imaginary parts of all \(\omega_n\) uniformly tend to \(-\infty\)). In other words, there is no non-trivial resonant frequency \(\omega(\delta)\) that depends continuously on \(\delta\) over \((0, +\infty)\). 
\end{example}

To proceed, we first recall some fundamental concepts from the theory of entire functions. Consider the following general trigonometric polynomial:
\begin{align}\label{equ: tri polynomial}
P(z) := \sum_{j=1}^n a_j \e^{\i \lambda_j z}, \quad \text{with } \lambda_1 < \lambda_2 < \cdots < \lambda_n\ \text{ and }\ 0 \neq a_j \in \mathbb{C}, \ 1 \leq j \leq n,
\end{align}
which is a canonical example of both \textit{almost-periodic functions} and \textit{entire functions of exponential type} \cite{levin1964zeros}.
For an almost-periodic function \( g \), one can define its mean value $\mathfrak{m}(g)$ and Fourier coefficients $a(\lambda)$, respectively, by 
\[
\mathfrak{m}(g) := \lim_{T \to +\infty} \frac{1}{2T} \int_{-T+\alpha}^{T+\alpha} g(t)  \mathrm{d}t, \q a(\lambda) := \mathfrak{m}\left( g(x) e^{-i\lambda x} \right), \quad \lambda \in \mathbb{R}, 
\]
where the convergence is uniform in \( \alpha \in \mathbb{R} \). 
These coefficients \(a(\lambda)\) are non-zero for at most a countable set of \(\lambda\). This set of \(\lambda\) constitutes the \textit{spectrum} of \(g\), denoted by \(\Lambda_g\). We refer the reader to \cite{levin1964zeros} for a detailed discussion. The following lemma, taken from \cite{levin1964zeros}, plays a crucial role in our analysis. 

\begin{lemma}
\label{lem:almost_periodic}
Let \( g \) be an entire almost-periodic function of exponential type.
\begin{enumerate}
    \item All the zeros of \( g \) lie in a horizontal strip parallel to the real axis if and only if the spectrum \( \Lambda_g \) satisfies
    \[
    \inf \Lambda_g \in \Lambda_g \quad \text{and} \quad \sup \Lambda_g \in \Lambda_g.
    \]
    
    \item Let \( m_g(x_1, x_2; y_1, y_2) \) denote the number of zeros of \( g \) in the rectangle \( [x_1, x_2] \times [y_1, y_2] \). If the spectrum \( \Lambda_g \) is bounded, then the linear density of zeros in a horizontal strip:
    \begin{align}\label{equ: m_g(y1,y2) def}
    m_g(y_1, y_2) := \lim_{x_2 - x_1 \to +\infty} \frac{m_g(x_1, x_2; y_1, y_2)}{x_2 - x_1},
    \end{align}
    satisfies
    \[
    \lim_{\substack{y_1 \to -\infty \\ y_2 \to +\infty}} m_g(y_1, y_2) = \frac{d}{2\pi},
    \]
   where \(d\) is the length of the smallest interval containing the spectrum set \(\Lambda_g\), or equivalently, $\sup \Lambda_g - \inf \Lad_g$. 
\end{enumerate}
\end{lemma}
With these concepts in hand, we now return to our main discussion. Example (i) shows that \(f(\cdot; 1)\) has no zeros, and thus (\ref{equ: scattering problem}) has only a trivial resonant frequency. However, when \(\sigma \ne 1\), this is not the case. There are countably many zeros of \(f(z; \sigma)\) when \(\sigma \ne 1\), which satisfy the following distribution property. In Figure \ref{fig: |f(k;sigma)|} below, we present a numerical simulation to verify the zero density formula \eqref{eq:zerodensity}. 

\begin{theorem}\label{thm: f(z;mu) zeros}
For fixed $\sigma \ne 1$, there exist constants $C_1(\sigma), C_2(\sigma) \in \mathbb{R}$ such that all zeros of $f(k;\sigma)$ satisfy
\[
C_1(\sigma) < \im k < C_2(\sigma).
\]
If $\delta \in \mathbb{R}_+$, Theorem \ref{thm: resonant frequency property}\,(2) allows us to set $C_2(\sigma) = 0$. Moreover, for any $x_1<x_2$, we define the number of zeros of $f(\cdot;\sigma)$ in the rectangle $[x_1,x_2] \times [C_1(\sigma), C_2(\sigma)]$: 
\[
n(x_1,x_2) = \#\left\{k \in \mathbb{C} \middle| f(k;\sigma) = 0,\ x_1 < \re k < x_2,\ C_1(\sigma) < \im k < C_2(\sigma)\right\}.
\]
Then, we have  
\begin{equation} \label{eq:zerodensity}
    \lim_{x_2 - x_1 \to +\infty} \frac{n(x_1,x_2)}{x_2 - x_1} = \frac{\|\bm{t}\|_1}{\pi}. 
\end{equation}
In particular, $f(\cdot;\sigma)$ possesses countably many zeros and consequently the scattering problem (\ref{equ: scattering problem}) admits countably many resonant frequencies.
\end{theorem}

\begin{proof}
When $\sigma \ne 1$, the expansion \eqref{eq: f(z; sigma) expansion} implies that at least two terms of $f(k;\sigma)$ are non-zero. Specifically, $f(k;\sigma)$ contains the terms $\e^{-\i\|\bm{t}\|_1 k}$ and $-\bigl(\tfrac{\sigma-1}{\sigma+1}\bigr)^{\!2} \e^{\i\|\bm{t}\|_1k}$, corresponding to the choices $\bm{\alpha} = \bm{1}$ and $\bm{\alpha} = -\bm{1}$, respectively. Furthermore, for any $\bm{\alpha} \in \{-1,1\}^{2N-1}$ with $\bm{\alpha} \ne \pm \bm{1}$, we have
\[
-\|\bm{t}\|_1 < \langle \bm{\alpha}, \bm{t} \rangle < \|\bm{t}\|_1.
\]
Thus we can rewrite $f$ in the form of (\ref{equ: tri polynomial}) with $\lambda_1 = -\|\bm{t}\|_1$ and $\lambda_n = \|\bm{t}\|_1$. The Fourier coefficients of $f$ can be computed as follows: 
\[
a(\lambda) = \frac{1}{2T} \sum_{j=1}^n a_j \lim_{T \to +\infty} \int_{-T+\alpha}^{T+\alpha} \e^{\i(\lambda_j - \lambda)t} \mathrm{d}t = 
\begin{cases}
    a_j, & \lambda = \lambda_j \text{ for some } j, \\
    0, & \text{otherwise}.
\end{cases}
\]
It follows that the spectrum of $f$ is given by $\Lambda_f = \{\lambda_1, \lambda_2, \cdots, \lambda_n\}$. 


By Lemma \ref{lem:almost_periodic}(1), there exist constants \(C_1(\sigma) < C_2(\sigma)\) such that all zeros of \(f\) lie within the strip \(C_1(\sigma) < \im k < C_2(\sigma)\). Let \(m_f(y_1, y_2)\) be defined as in \eqref{equ: m_g(y1,y2) def}. Then, we have \(m_f(y_1, y_2) = m_f(C_1(\sigma), C_2(\sigma))\) for all \(y_1 \leq C_1(\sigma)\) and \(y_2 \geq C_2(\sigma)\).  
Applying Lemma \ref{lem:almost_periodic}(2), we obtain 
\[
m_f(C_1(\sigma), C_2(\sigma)) = \lim_{\substack{y_1 \to -\infty \\ y_2 \to +\infty}} m_f(y_1, y_2) = \frac{\lambda_n - \lambda_1}{2\pi} = \frac{\|\bm{t}\|_1}{\pi}.
\]
Consequently, we have \eqref{eq:zerodensity} by 
\begin{align*}
    \lim_{x_2 - x_1 \to +\infty} \frac{n(x_1, x_2)}{x_2 - x_1} = \lim_{x_2 - x_1 \to +\infty} \frac{m_f(x_1, x_2; C_1(\sigma), C_2(\sigma))}{x_2 - x_1} = m_f(C_1(\sigma), C_2(\sigma)). 
\end{align*}
\end{proof}

\begin{remark}
Theorem \ref{thm: f(z;mu) zeros} reveals a fundamental distinction between the one-dimensional and three-dimensional problems. In three dimensions, the bounds on the imaginary part of the resonant frequencies generally depend on its real part and no single (uniform) line can bound all resonances from below; see \cite{li2025highcontrasttransmissionfabryperottype} for the resonance-free region. In the one-dimensional case, however, we prove the existence of such a uniform lower line for all the resonant frequencies. This result also identifies a so-called resonance-free region, which is of great theoretical interest \cite{dyatlov2019mathematical,moiola2019acoustic}.   
\end{remark}
\begin{figure}[htbp!]
    \centering
    \begin{subfigure}{0.42\textwidth}
        \centering
        \includegraphics[width=\linewidth]{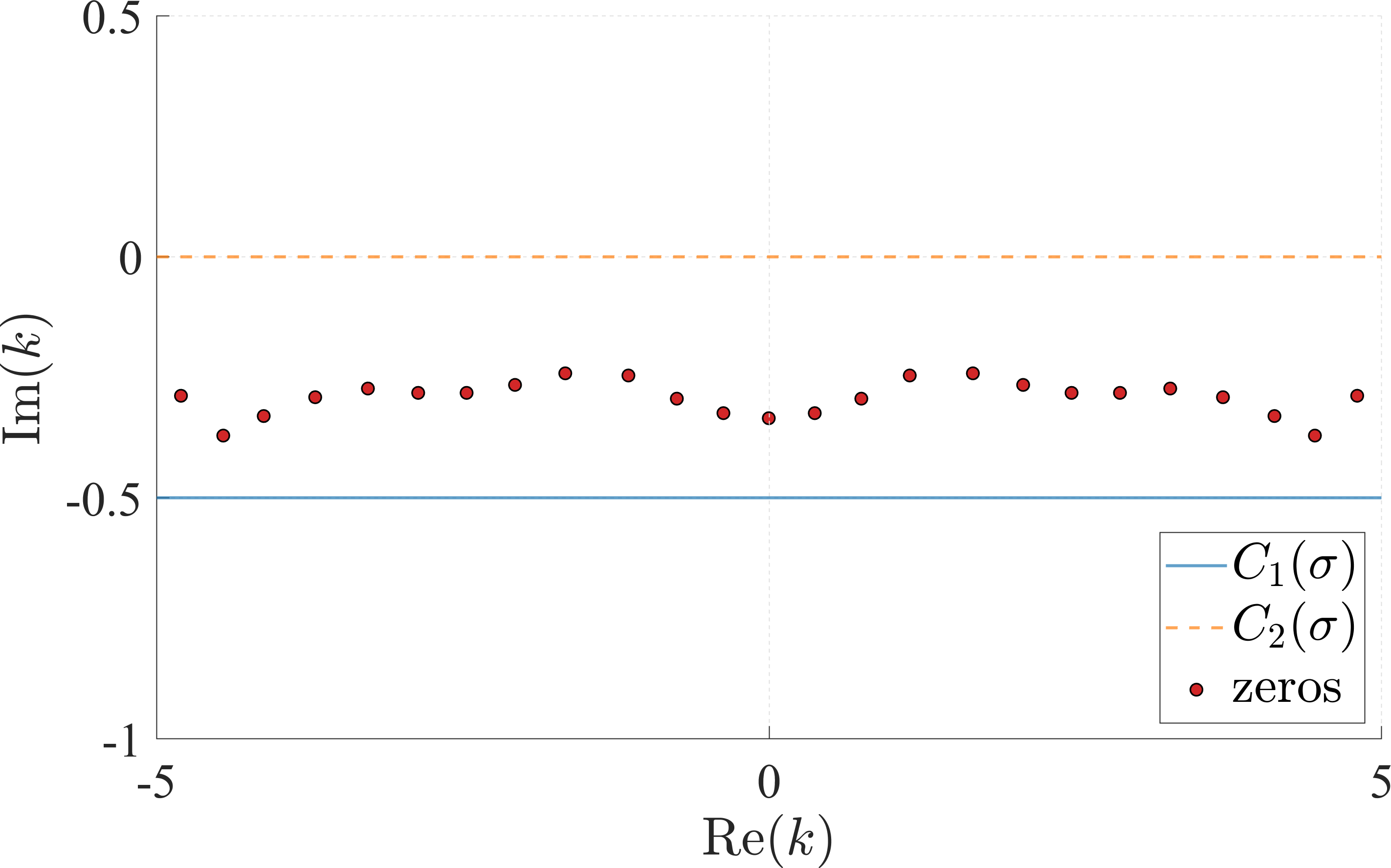}
        \caption{Zeros of $f(k;\sigma)$ with \(|\re k| \leq 5\). }
        \label{fig:zeros_region_5}
    \end{subfigure}
    \begin{subfigure}{0.42\textwidth}
        \centering
        \includegraphics[width=\linewidth]{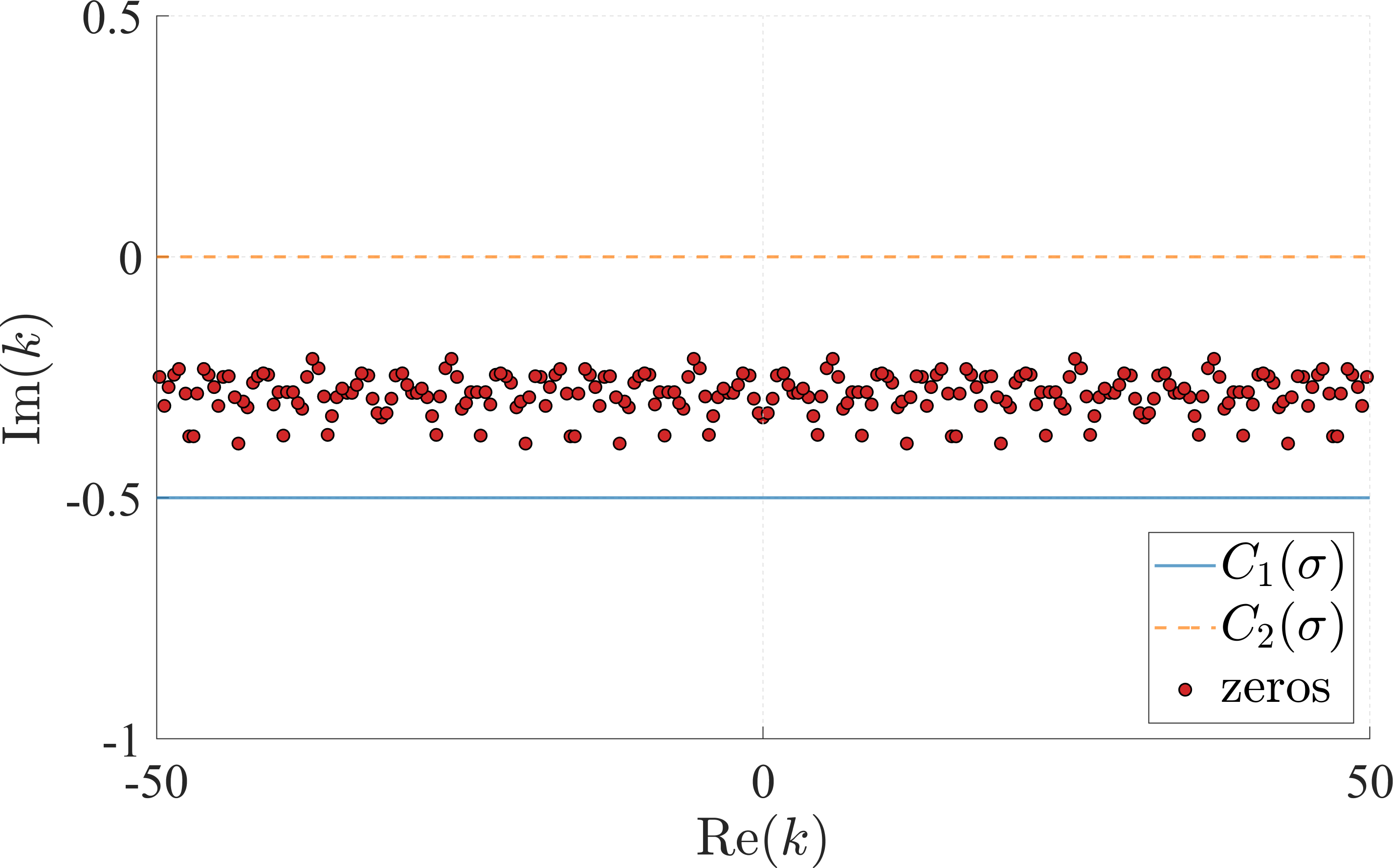}
        \caption{Zeros of $f(k;\sigma)$ with \(|\re k| \leq 50\). }
        \label{fig:zeros_region_50}
    \end{subfigure}
    \caption{Zeros of \(f(k; \sigma)\) for the configuration \(\bm{t} = (0.8, 0.9, 1, 1.1, 1.2, 1.3, 1.4)^\top\) and \(\sigma = 0.8\). All zeros are confined to the strip \(C_1(\sigma) < \im k < C_2(\sigma)\). In the region \(|\re k| \leq 5\) and \(|\re k| \leq 50\), the argument principle yields 25 zeros and 245 zeros, giving a density of 2.5 and 2.45 zeros per unit length, respectively. These empirical densities align closely with the theoretical one \(\|\bm{t}\|_1 / \pi \approx 2.451\) established in Theorem~\ref{thm: f(z;mu) zeros}.}
    \label{fig: |f(k;sigma)|}
\end{figure}

\subsection{Limiting cases of \texorpdfstring{$\delta\to0$}{δ→0} and \texorpdfstring{$\delta\rightarrow \infty$}{δ→∞} } 
We will characterize the limiting distribution of resonant frequencies of (\ref{equ: scattering problem}) when $\d\to0$ and $\d\to\infty$. 

From the definition \eqref{def:rlmatrix}, we first obtain 
\begin{align}\label{equ: R(sigma) limit}
\lim_{\sigma\to0}\frac{2}{1+\s}R(\s)=R_+:=\begin{pmatrix}
    1&1\\1&1
\end{pmatrix},\quad \lim_{\sigma\to0}\frac{2\s}{1+\s}R\left(\frac{1}{\s}\right)=R_-:=\begin{pmatrix}
    1&-1\\-1&1
\end{pmatrix}.
\end{align}
Thus, as $\s\to0$, the matrix $M_{tot}(k;\s)$ in (\ref{equ: M_total def}) uniformly converges to
\begin{align*}
M_{tot}(k;0):=&R_-L(t_{2N-1}k)R_+L(t_{2N-2}k)R_-\cdots R_-L(t_1k)R_+\\
=&(2\i)^{2N-1}\prod_{j=1}^{2N-1}\sin (t_jk)\cdot\begin{pmatrix}
    1&1\\-1&-1
\end{pmatrix},
\end{align*}
on any compact set $K\subset\C$. 
We then have the following lemma.
\begin{lemma}\label{thm: f(k;0) form}
When $\sigma = 0$, the analytic function $f(k;0)$ in \eqref{eq: f(z; sigma) expansion} becomes 
\begin{align*}
    f(k;0)=-(2\i)^{2N-1}\prod_{j=1}^{2N-1}\sin (t_j k).
\end{align*}
Then $k\in \C$ is a zero of $f(\cdot,0)$ if and only if $k\in E: = \cup_{j=1}^{2N-1}(\pi \Z/t_j)$. When $k\in E$, $k$ is a zero of order
\[
n(k):=\#\left\{j\middle|t_j k\in \pi\Z,1\leq j\leq2N-1\right\}.
\]
In particular, all zeros of $f(\cdot;0)$ are real and $0$ is a zero of order $2N-1$.
\end{lemma}

Since $f(k;\s)$ uniformly converges to $f(k;0)$ on any compact set as $\s\to0$, the Rouch\'e's Theorem implies that there are exactly $n(k)$ zeros of $f(\cdot;\s)$ in a small neighborhood of $k$ for every $k\in E$ when $|\s|$ is small enough. This shows that (\ref{equ: scattering problem}) has exactly $n(k)$ non-trivial resonant frequencies near $kv$ for every $k\in E$ when $\d\to0$. Moreover, by the Newton polygon method as in the next section, we can deduce that \(\omega(\delta)\) is an analytic function of \(\delta^s\) for certain rational \(s\) satisfying $s\geq\frac{1}{n(k)}$. According to Theorem \ref{thm: resonant frequency property}, similar conclusions hold as \(\delta \to \infty\). We summarize our discussion as follows. Figure \ref{fig:zeros_of_f(k;sigma)_with_delta_small} provides a numerical illustration of Theorem \ref{thm: zero near k}. 

\begin{theorem}\label{thm: zero near k}
For $\delta$ small enough (or large enough), 
there exist exactly $n(k^*)$ (counting multiplicities) non-trivial resonant frequencies $\omega = kv$ near $k^*v$, where $k^*\in \cup_{j=1}^{2N-1}\pi\Z/t_j$, with $C_1(\s)<\im k<C_2(\s)$ 
for $C_1(\s), C_2(\s)$ given as in Theorem \ref{thm: f(z;mu) zeros}.  In particular, near each $k^*v$, the resonance $\omega(\delta)$ is an analytic function of $\d^s$ (or $\d^{-s}$, respectively) for certain rational $s\geq\frac{1}{n(k^*)}$.
\end{theorem}

\begin{figure}[htbp!]
    \centering
    
    \begin{subfigure}{\textwidth}
        \centering
        \includegraphics[width=0.95\textwidth]{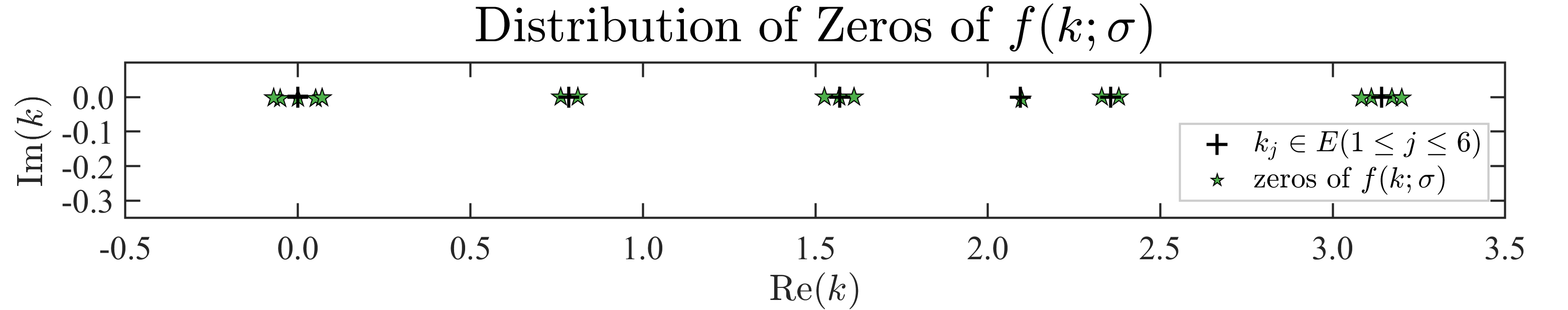}
        \caption{Zeros of \(f(k; \sigma)\) with  $-0.5 \le \re k \le 3.5$, under the configuration \(N=3\), \(\delta=0.01\), \(r=1\), \(\ell_1=1.5\), \(\ell_2=4\), \(\ell_3=1\), \(s_1=2\), and \(s_2=4\). There are $6$ points in $E$ satisfying $-0.5 \le k \le 3.5$, i.e., $k_1=0,k_2=\pi/4\approx0.7854,k_3=\pi/2\approx1.5708,k_4=2\pi/3\approx2.0944,k_5=3\pi/4\approx2.3562,k_6=\pi\approx3.1416$.}
        \label{fig:zeros_of_f(k;sigma)}
    \end{subfigure}

    \begin{subfigure}{0.327\textwidth}
        \centering
        \includegraphics[width=\linewidth]{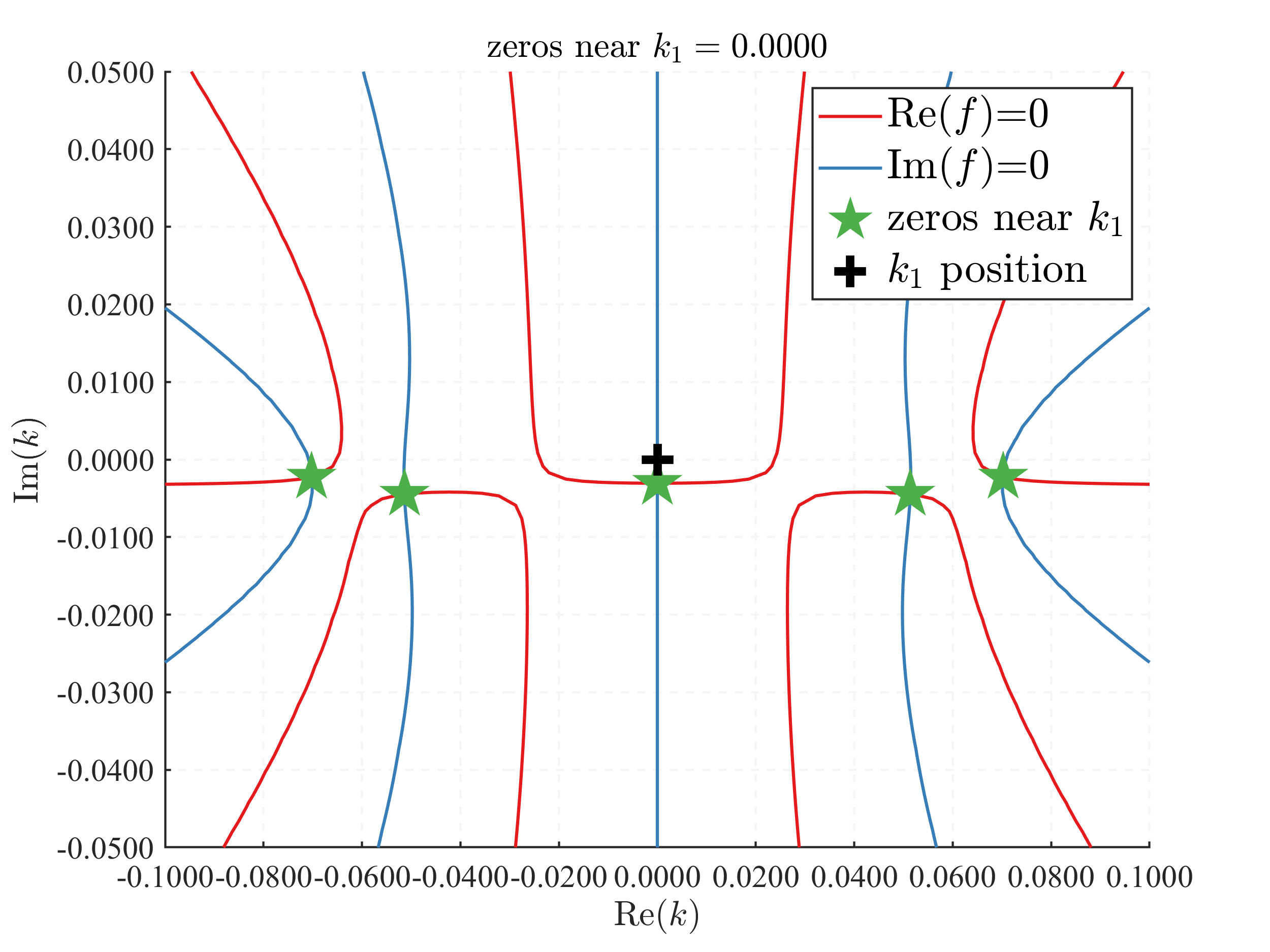}
        \caption{$n(k_1)=5$}
        \label{fig:zeros_near_k1}
    \end{subfigure}
    \begin{subfigure}{0.327\textwidth}
        \centering
        \includegraphics[width=\linewidth]{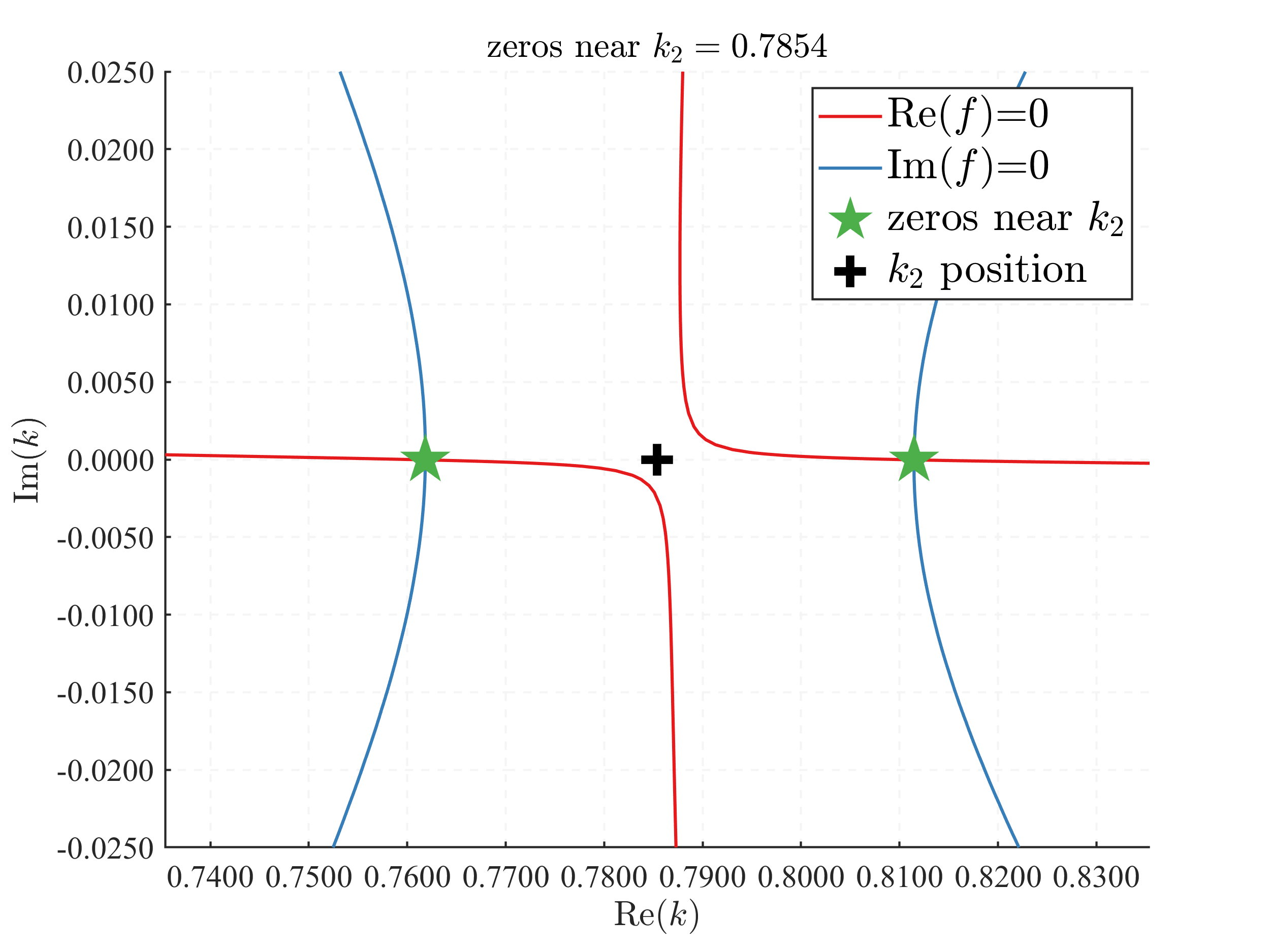}
        \caption{$n(k_2)=2$}
        \label{fig:zeros_near_k2}
    \end{subfigure}
    \begin{subfigure}{0.327\textwidth}
        \centering
        \includegraphics[width=\linewidth]{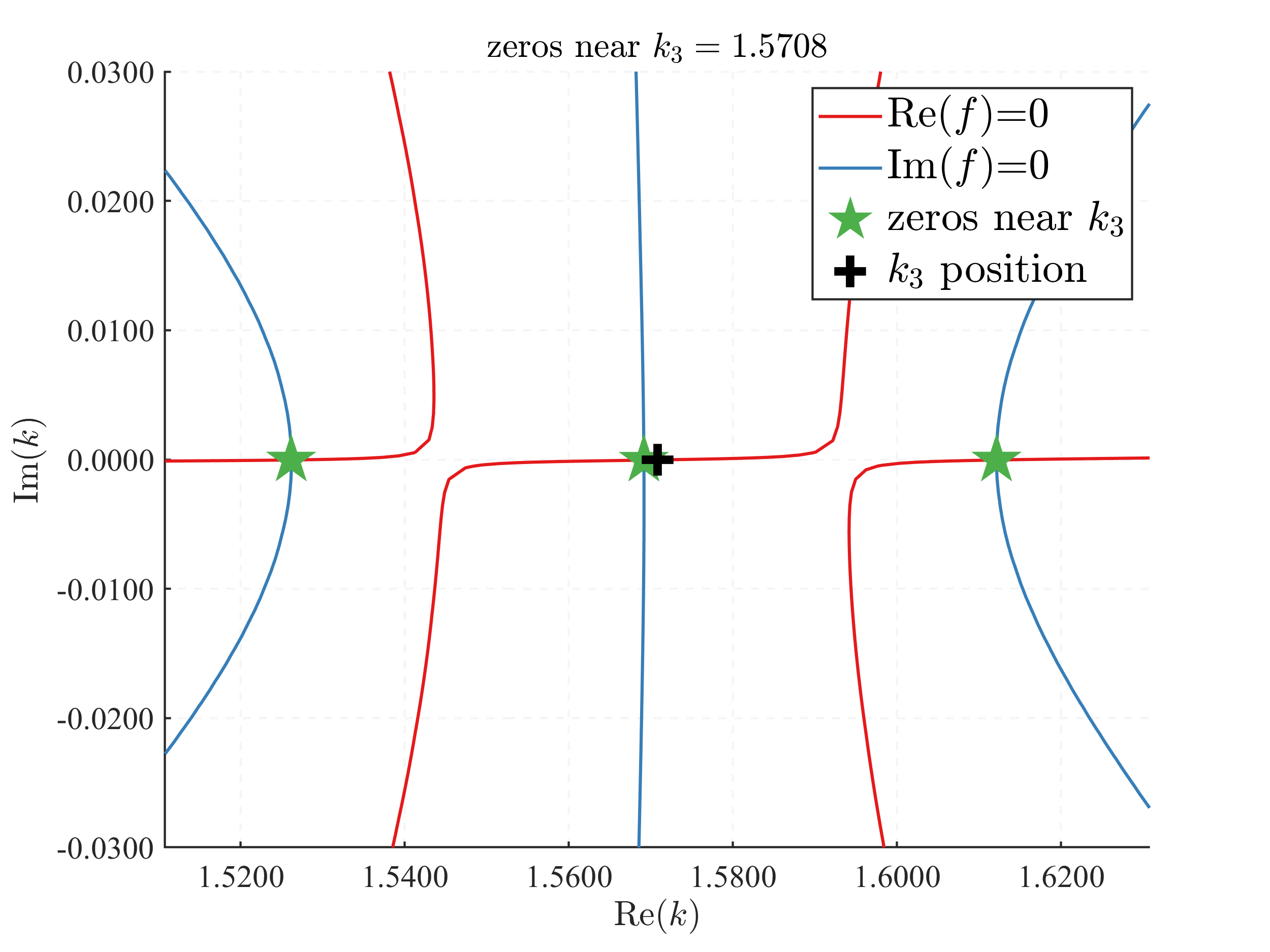}
        \caption{$n(k_3)=3$}
        \label{fig:zeros_near_3}
    \end{subfigure}
    \vspace{0.0cm}
    \begin{subfigure}{0.327\textwidth}
        \centering
        \includegraphics[width=\linewidth]{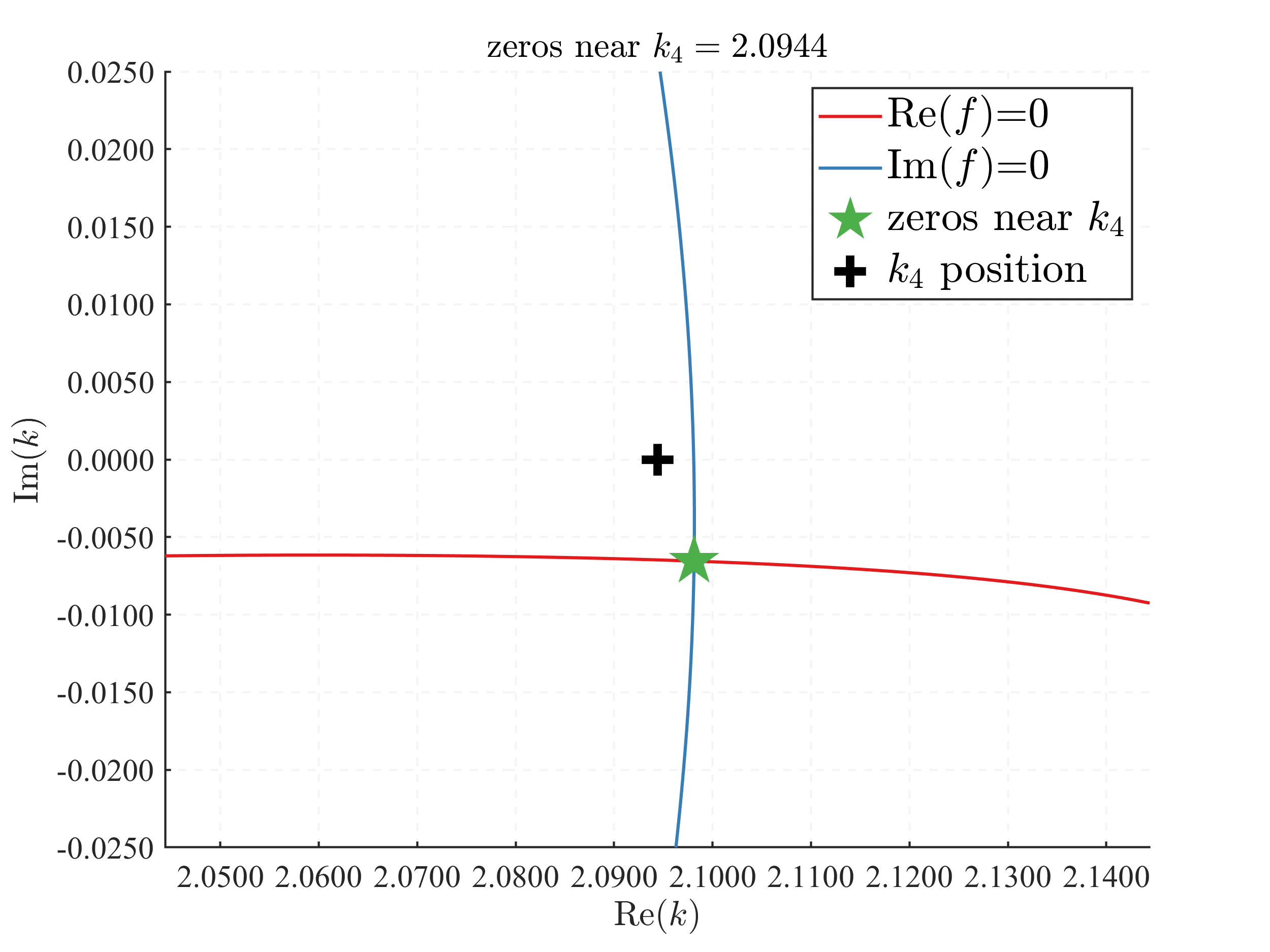}
        \caption{$n(k_4)=1$}
        \label{fig:zeros_near_k4}
    \end{subfigure}
    \begin{subfigure}{0.327\textwidth}
        \centering
        \includegraphics[width=\linewidth]{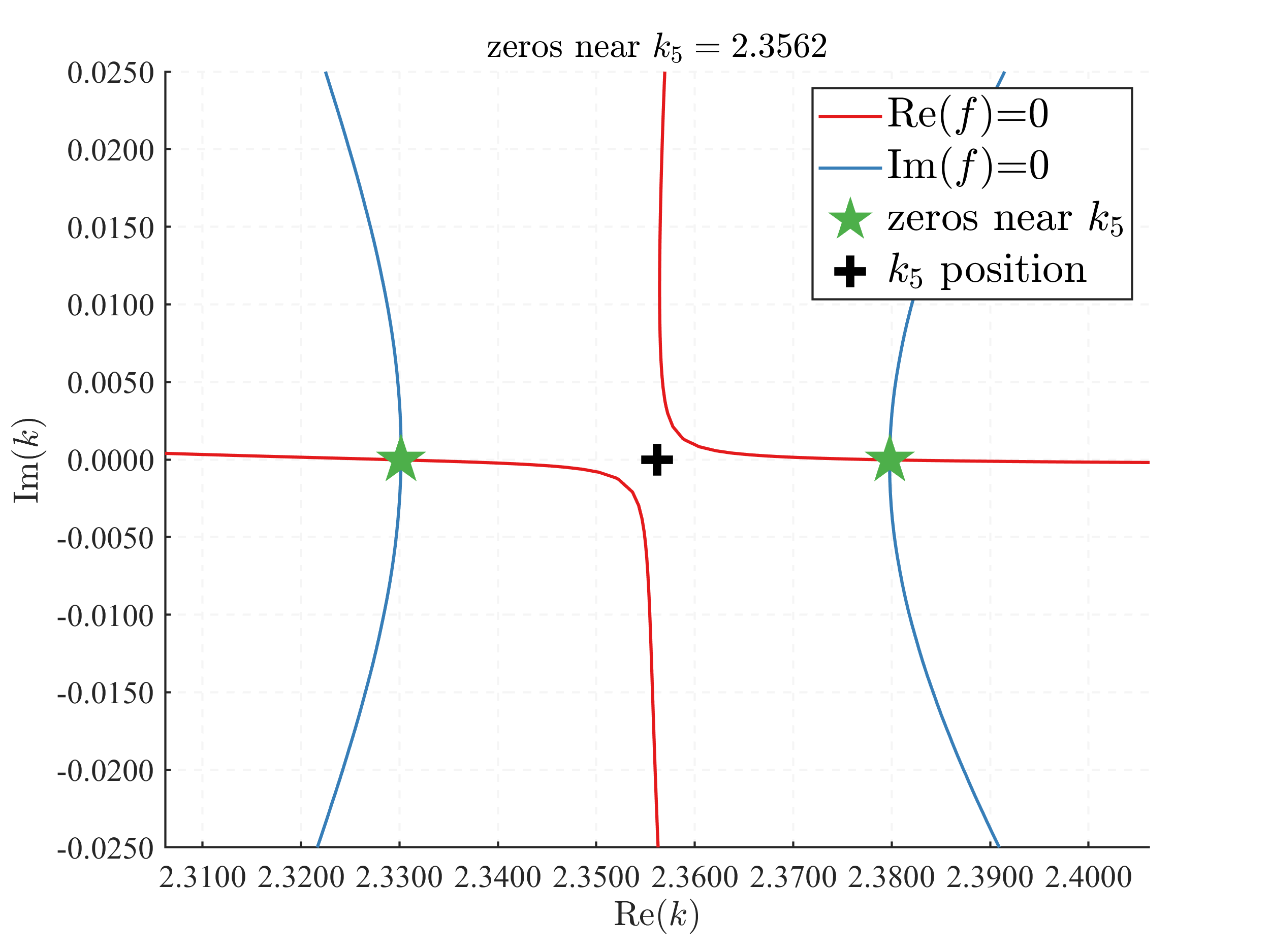}
        \caption{$n(k_5)=2$}
        \label{fig:zeros_near_k5}
    \end{subfigure}
    \begin{subfigure}{0.327\textwidth}
        \centering
        \includegraphics[width=\linewidth]{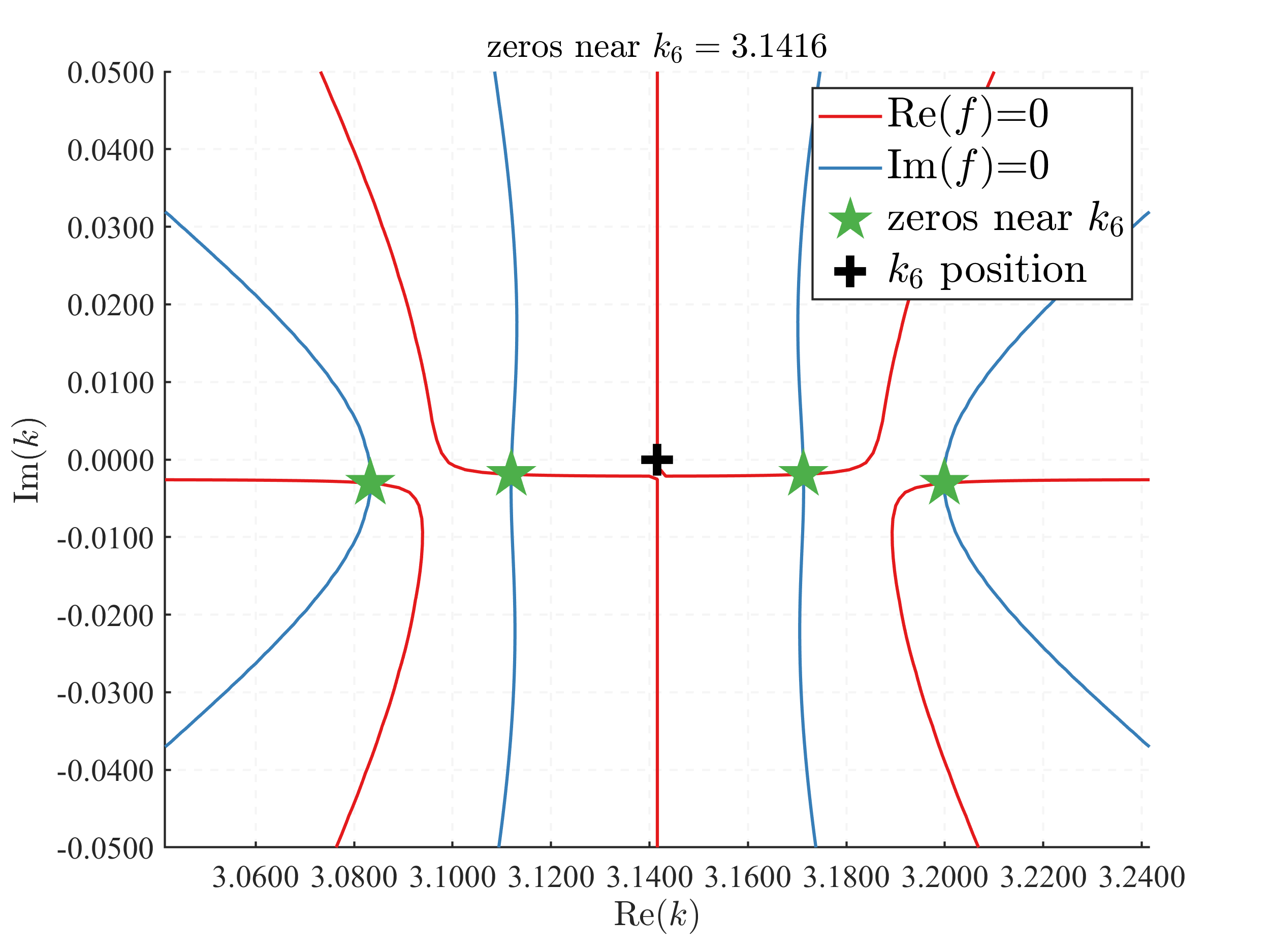}
        \caption{$n(k_6)=4$}
        \label{fig:zeros_near_k6}
    \end{subfigure}
    
    \caption{Zeros of $f(k; \sigma)$ with sufficiently small $\delta$. Panels (b)-(g): Zoom-in views near \(k_1\)-\(k_6\) with different orders \(n(k_i)\). We see that exactly $n(k)$ zeros are located near each $k \in E$, where the set $E$ is given in Lemma \ref{thm: f(k;0) form}.}\label{fig:zeros_of_f(k;sigma)_with_delta_small}
\end{figure}

\subsection{Asymptotic expansions of resonances}
We next analyze the asymptotic behaviors of resonances as \(\delta \to 0\) and \(\delta \to \infty\). 
For simplicity, we focus on the case where \(n(k) = 1\). Notably, this already covers the case in the concurrent work \cite{li2025highcontrasttransmissionfabryperottype}, where asymptotic expansions of resonances are derived for a single resonator in three dimensions. Compared to their results, our asymptotic formulas for the one-dimensional case are more explicit. For cases where $n(k_0) > 1$, asymptotic analysis can in principle be carried out using the Newton polygon method, as presented in the next section for subwavelength resonances. However, such computations are operationally cumbersome and will be reported in a forthcoming work.


For notational simplicity, we will sometimes use $L_j(k)$ to denote $L(t_j k)$. Let
\begin{equation} \label{def:matrixrs}
    \nu:=\frac{2\sigma}{1+\sigma}, \qquad R:=\begin{pmatrix}
    -1 & 1 \\ -1 & 1
\end{pmatrix}, \qquad S:=\begin{pmatrix}
    0 & -1 \\ 1 & 0
\end{pmatrix}. 
\end{equation}
For the components $\tfrac{4\s}{(1+\s)^2}R\left(\tfrac{1}{\s}\right)L(k)R(\s)$ in $M_{tot}(k;\s)$, we observe that 
\[
\frac{4\s}{(1+\s)^2}PR\left(\frac{1}{\s}\right)L(k)R(\s)P=(R+\nu S)L(k)(R+\nu S),\q \text{with}\ P=\begin{pmatrix}
    -1&&\\
    &&1
\end{pmatrix}.
\]
Therefore, it holds that
\[
PM_{tot}(k;\s)P= G(k;\nu), \quad f(k;\s) = g(k;\nu),
\]
where $G(k;\nu)$ and $g(k;\nu)$ are defined by 
\begin{align} \label{equ: Matrix rep4}
    &G(k;\nu):=(R+\nu S)L_{2N-1}(k)(R+\nu S)L_{2N-2}(k)(R+\nu S)\cdots L_{1}(k)(R+\nu S),\\ \label{equ:Non_Matrix rep3}
    &g(k;\nu):=G(k,\nu)_{2,2}.
\end{align}
For the case when $\delta \rightarrow 0$, from (\ref{equ: Matrix rep4}), we can expand $G(k;\nu)$ and $g(k;\nu)$ as
\begin{equation} \label{equ:G_expession}
    \begin{dcases}
    G(k;\nu):=G_0(k)+G_1(k)\nu+G_2(k)\nu^2+\cdots+G_{2N}(k)\nu^{2N},\\ 
    g(k;\nu):=g_0(k)+g_1(k)\nu+g_2(k)\nu^2+\cdots+g_{2N}(k)\nu^{2N}.
    \end{dcases}
\end{equation}

Using the following easily verifiable identities, with $R_{\pm}$, $R$ and $S$ defined in (\ref{equ: R(sigma) limit}) and \eqref{def:matrixrs}, 
\begin{equation}\label{equ:matrixexpan1}
\begin{aligned}
        &RL_j(k)R = -2\i\sin(t_jk) R\,,\q RL_{j+1}(k)SL_{j}(k)R= 2\cos((t_{j+1}-t_{j})k) R\,,\\
        &SL_{2N-1}(k)R=L_{2N-1}(-k)R_{-}\,, \quad RL_{1}(k)S=R_{+}L_1(-k)\,,
\end{aligned}
\end{equation}
we can explicitly compute $G_1(k)$ in the expansion \eqref{equ:G_expession} as 
\begin{align} \label{equ: G_1(k) rep}
    G_1(k)= & (2\i)^{2N-2}\prod_{j=1}^{2N-2}\sin (t_j k) L(-t_{2N-1}k)R_-+(2\i)^{2N-2}\prod_{j=2}^{2N-1}\sin (t_jk) R_+L(-t_1k) \notag \\
&-2(2\i)^{2N-3}\sum_{j=1}^{2N-2}\cos((t_{j+1}-t_j)k)\prod_{\substack{s=1\\s\ne j,j+1}}^{2N-1}\sin (t_s k) R.
\end{align}

We are now in a position to derive the asymptotic expansions of the resonant frequencies for wavenumbers $k_0$ satisfying $n(k_0) = 1$. Figure \ref{fig:asymptotic_behavior_when_n(k)=1} provides a numerical illustration of Theorem \ref{thm:resonancesexpan1}.

\begin{theorem}\label{thm:resonancesexpan1}
Assume $n(k_0)=1$. When $\d\to0$ (or $\d\to\infty$), the scattering problem (\ref{equ: scattering problem}) has a unique resonance $\ww(\d)$ in a neighborhood of $k_0v$, which is an analytic function of $\d$ (or $\d^{-1}$, respectively) with the first-order asymptotics:
\begin{align}\label{equ: asymptotic_behavior_of_omega(delta)}
\ww(\delta)=  k_0v + \begin{dcases}
   \frac{\cot (s_1k_0)-\i }{r^2\ell_1}v\delta+O(\delta^2),&\pi\mid r\ell_1k_0,\\
       \frac{\cot (s_{N-1}k_0)-\i }{r^2\ell_N}v\delta+O(\delta^2),&\pi\mid r\ell_Nk_0,\\
   \frac{\cot( s_{j-1}k_0)+\cot (s_jk_0)}{r^2\ell_j}v\delta+O(\delta^2),&\pi\mid r\ell_jk_0,\ 1<j<N,\\
   \frac{\cot (r\ell_j k_0)+\cot (r\ell_{j+1}k_0)}{rs_j}v\delta+O(\delta^2),&\pi\mid s_jk_0,\ 1\leq j\leq N-1,
\end{dcases}
\end{align}
as $\d\to0$, and as $\d\to\infty$,
\[
\ww(\delta)= k_0v + \begin{dcases}
    \frac{\cot (s_1k_0)-\i }{\ell_1}\frac{v}{\delta}+O(\delta^{-2}),&\pi\mid r\ell_1k_0,\\
    \frac{\cot (s_{N-1}k_0)-\i }{\ell_N}\frac{v}{\delta}+O(\delta^{-2}),&\pi\mid r\ell_Nk_0,\\
    \frac{\cot (s_{j-1}k_0)+\cot (s_jk_0)}{\ell_j}\frac{v}{\delta}+O(\delta^{-2}),&\pi\mid r\ell_jk_0,\ 1<j<N,\\
    \frac{\cot (r\ell_j k_0)+\cot (r\ell_{j+1}k_0)}{s_j}\frac{vr}{\delta}+O(\delta^{-2}),&\pi\mid s_jk_0,\ 1\leq j\leq N-1.
\end{dcases}
\]
\end{theorem}


\begin{proof}
We will focus on proving the case \(\delta \to 0\); the case \(\delta \to \infty\) then follows straightforwardly from Theorem \ref{thm: resonant frequency property}. Moreover, noting that \(\nu = \tfrac{2\sigma}{1+\sigma} = \tfrac{2\delta/r}{1+\delta/r}\) and the relation \(\omega = k v\), it suffices to derive the expansion of \(k(\nu)\) in terms of \(\nu\). 

 According to Lemma \ref{thm: f(k;0) form}, we have 
 \begin{equation*}
      g(k;0)=-(2\i)^{2N-1}\prod_{j=1}^{2N-1}\sin(t_j k).
 \end{equation*}
\(n(k_0) = 1\) implies that there exists an \(n \in \mathbb{Z}\) and \(1 \leq j \leq 2N-1\) such that \(t_jk_0 = n\pi\), while \(\pi \nmid t_sk_0\) for any \(s \neq j\). Thus, we derive 
\begin{align*}
     g(k_0;0)=0,\quad \frac{\p g}{\p k}(k_0;0)=(-1)^{n+1}(2\i)^{2N-1}t_j\prod_{\substack{s=1\\s\ne j}}^{2N-1}\sin (t_s k_0) \ne 0.
\end{align*}
By the implicit function theorem, there exists a unique function \(k(\nu)\) defined in a neighborhood of \(\nu = 0\), which is analytic in \(\nu\) and satisfies \(g(k(\nu); \nu) = 0\) and \(k(0) = k_0\).  Moreover, since
 \[
 \frac{\p g}{\p\nu}(k;0)=g_1(k)=G_1(k)_{2,2},
 \]
 the expansion (\ref{equ: G_1(k) rep}) gives 
\begin{multline*}
    \frac{\p g}{\p \nu}(k_0;0) = \\ 
  (-1)^{n+1}2(2\i)^{2N-3} \cdot  \begin{dcases}
\e^{-\i  t_2k_0}\prod_{s=3}^{2N-1}\sin (t_sk_0), & j = 1,\\
\sin((t_{j-1}+t_{j+1})k_0)\prod_{\substack{s=1\\|s-j|>1}}^{2N-1}\sin (t_sk_0),& 1<j<2N-1,\\
\e^{-\i t_{2N-2}k_0}\prod_{s=1}^{2N-3}\sin (t_sk_0), &j = 2N -1.
 \end{dcases}
\end{multline*}
 It follows that
 \[
 k'(0)=-\frac{\frac{\partial g}{\partial \nu}(k_0;0)}{\frac{\partial g}{\partial k}(k_0;0)}=\begin{dcases}
     \frac{\cot (t_2k_0)-\i}{2t_1},&j=1,\\
     \frac{\cot (t_{j-1} k_0) + \cot (t_{j+1}k_0)}{2t_j},&1<j<2N-1,\\
     \frac{\cot (t_{2N-2}k_0)-\i}{2t_{2N-1}},&j=2N-1.
 \end{dcases}
 \]
Substituting this into the expansion \(k(\nu) = k_0 + k'(0)\nu + O(\nu^2)\), along with the definition of the vector \(\bm t\) in \eqref{def:vectort}, completes the proof.
\end{proof}


\begin{figure}[hbtp!]
    \centering
    \begin{subfigure}{0.45\textwidth}
        \centering
        \includegraphics[width=\linewidth]{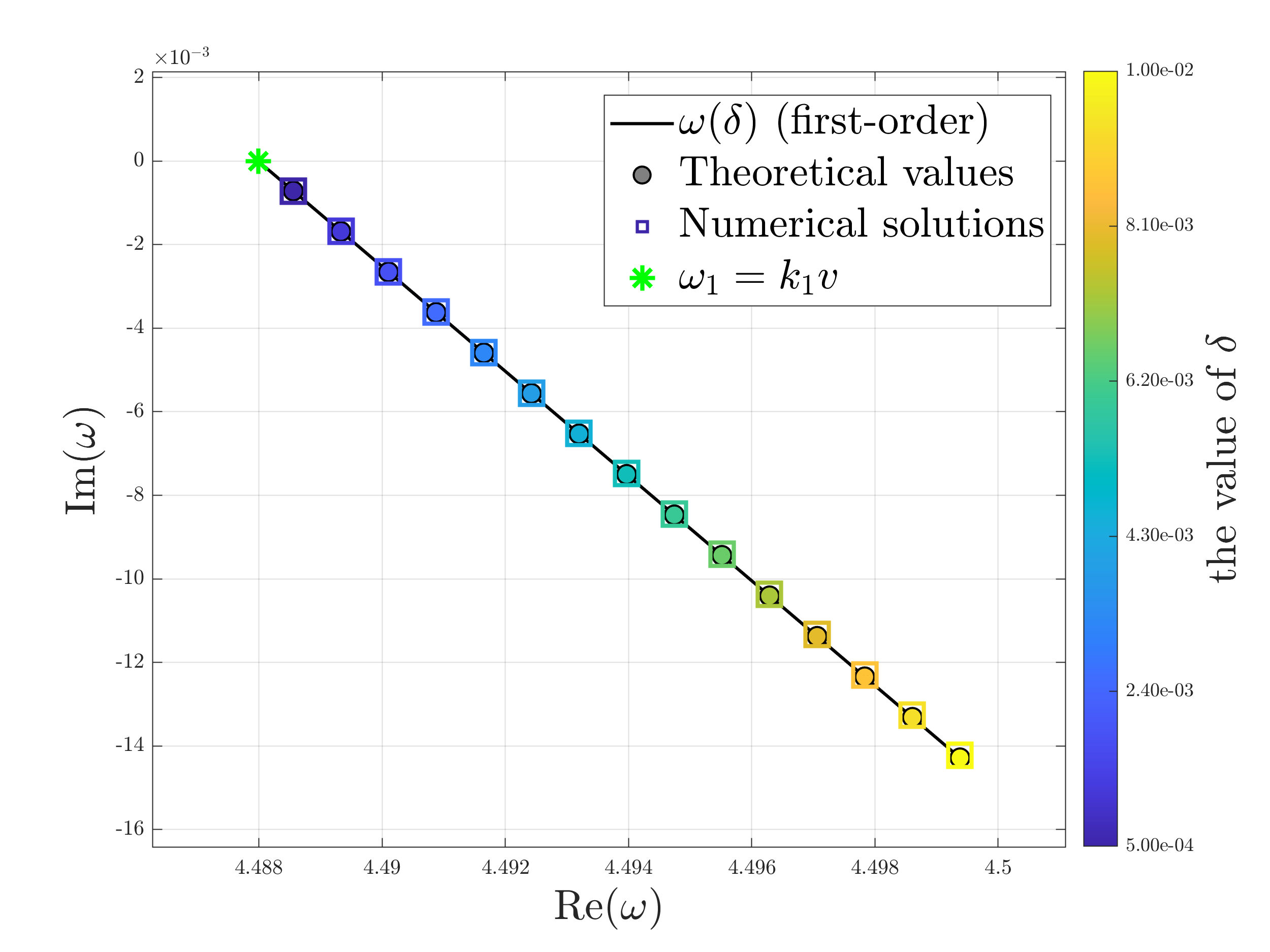}
        \caption{$k_1 = \tfrac{\pi}{0.3} \approx 10.4720$}
        \label{fig:asymoptic_near_k1}
    \end{subfigure}
    \hfill
    \begin{subfigure}{0.45\textwidth}
        \centering
        \includegraphics[width=\linewidth]{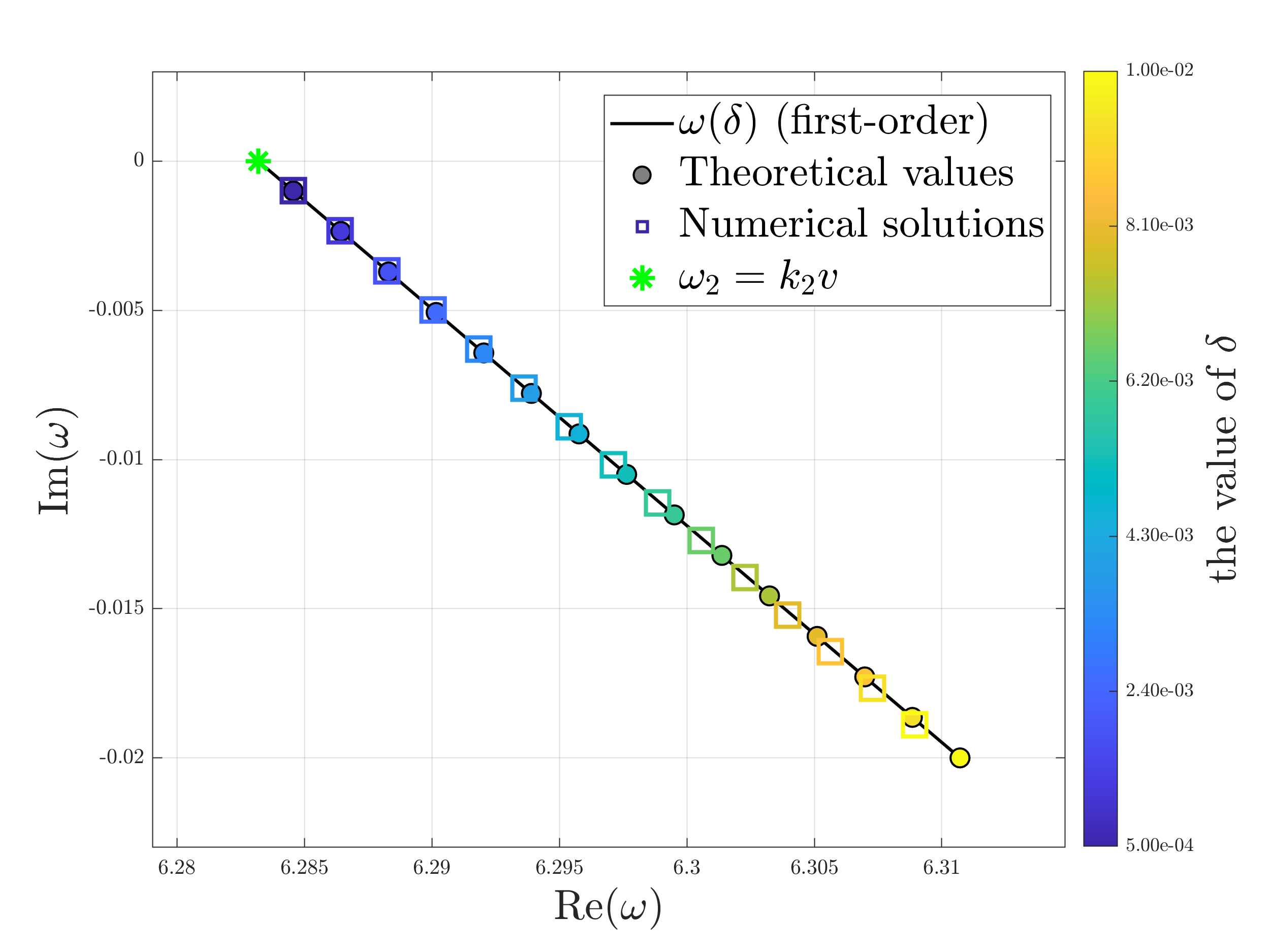}
        \caption{$k_2 = \tfrac{\pi}{0.5} \approx 6.2832$}
        \label{fig:asymoptic_near_k2}
    \end{subfigure}
    \begin{subfigure}{0.45\textwidth}
        \centering
        \includegraphics[width=\linewidth]{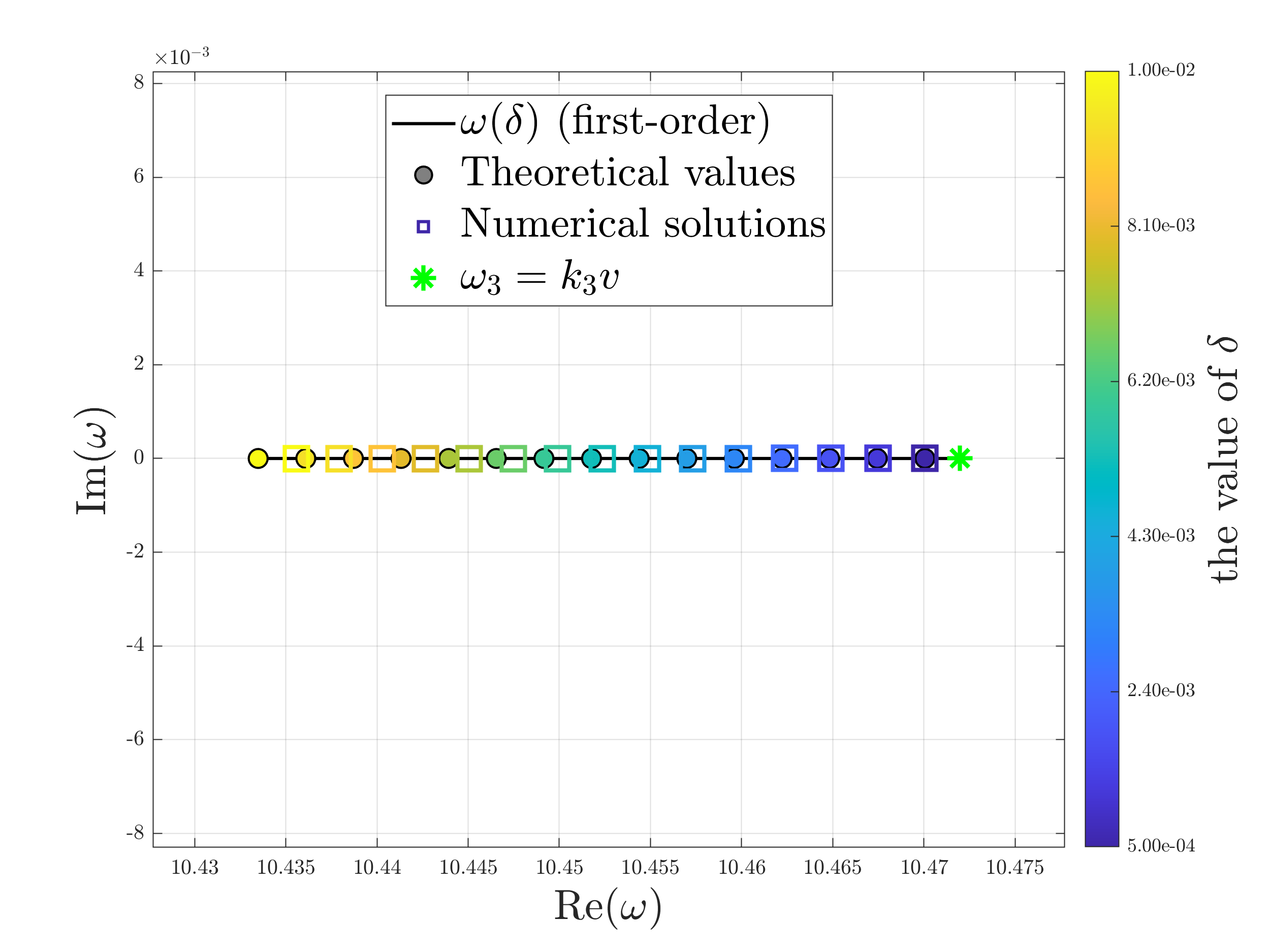}
        \caption{$k_3 = \tfrac{\pi}{0.7} \approx 4.4880$}
        \label{fig:asymptotic_near_k3}
    \end{subfigure}
     \hfill
    \begin{subfigure}{0.45\textwidth}
        \centering
        \includegraphics[width=\linewidth]{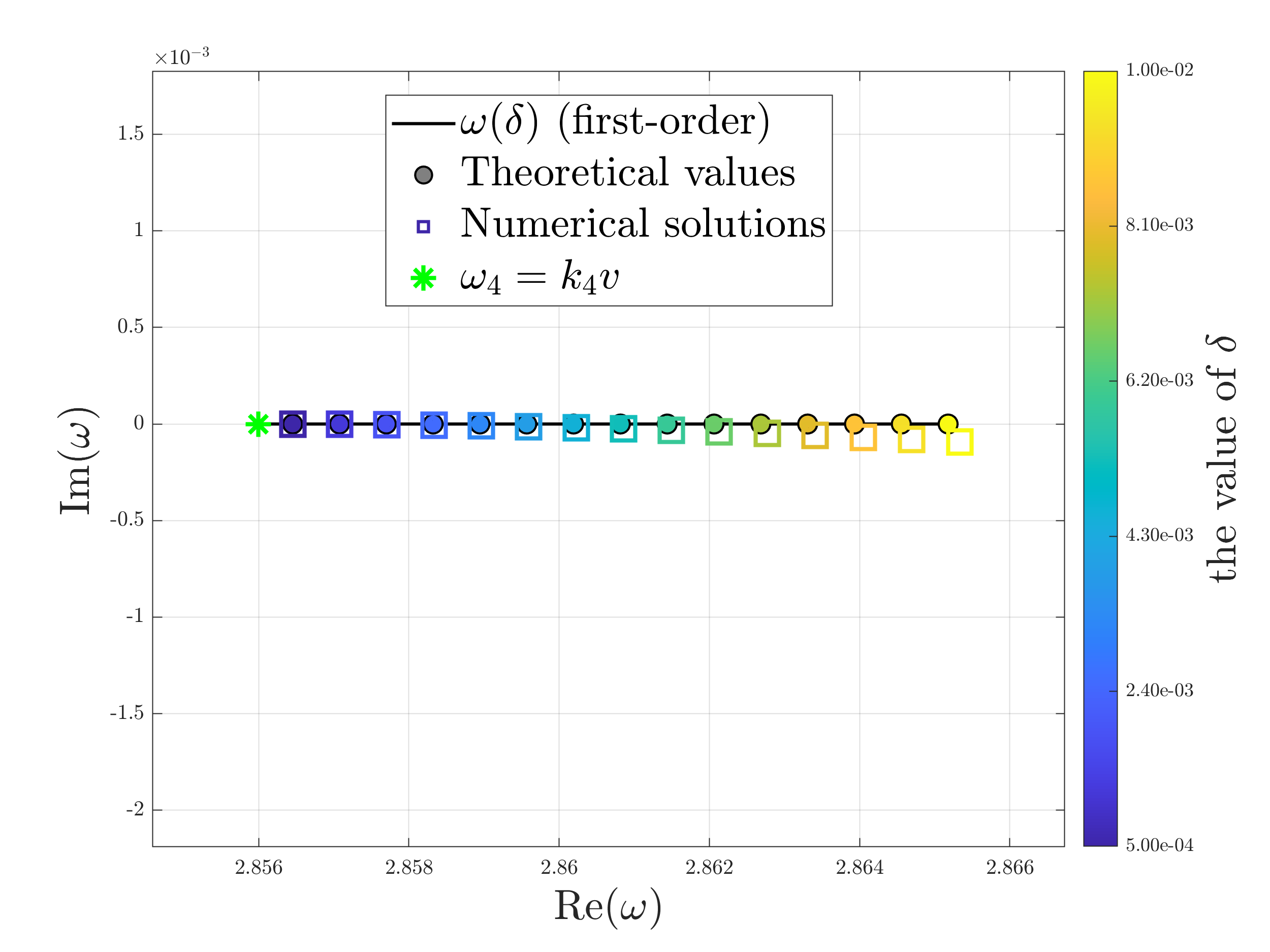}
        \caption{$k_4 = \tfrac{\pi}{1.1} \approx 2.9560$}
        \label{fig:asymptotic_near_k4}
    \end{subfigure}
    \caption{Asymptotic behavior of \(\omega(\delta)\) for sufficiently small \(\delta\) is analyzed near \(\omega_j = k_j v\) with \(n(k_j) = 1\) for \(1 \leq j \leq 4\), under the configuration \(N = 3\), \(r = 1\), \(v = 1\), \(\ell_1 = 0.7\), \(\ell_2 = 0.3\), \(\ell_3 = 0.5\), \(s_1 = 0.2\), and \(s_2 = 1.1\). 
    Panels (a)-(d) correspond to the first through fourth cases in \eqref{equ: asymptotic_behavior_of_omega(delta)}, respectively.}
    \label{fig:asymptotic_behavior_when_n(k)=1}
\end{figure}

\section{Subwavelength resonances and capacitance matrix theory}\label{sec:capacitancematrixtheroy1}

In this section, we study the resonant frequencies in the subwavelength regime; see Definition \ref{defi:subwavelengthresonance}. Specifically, we recover the capacitance matrix theory in \cite{Subwavelength_1D_High-Contrast} for subwavelength resonances using the propagation matrix approach, revealing the relationships between the two frameworks. Our new method is conceptually distinct from the variational approach in \cite{feppon2024subwavelength, Subwavelength_1D_High-Contrast}, which is based on Dirichlet-to-Neumann maps, and is more natural for the one-dimensional problem.


\begin{definition}\label{defi:subwavelengthresonance}
We call \(\omega(\delta) \in \C\) a \textit{subwavelength resonant frequency} if \(\omega(\delta)\) depends continuously on \(\delta \in \C\) for sufficiently small (or large) \(\delta\) and satisfies \(\omega(\delta) \to 0\) as \(\delta \to 0\) (or \(\delta \to \infty\)). 
\end{definition}

\subsection{Subwavelength eigenfrequencies}

In this section, we analyze the asymptotic behavior of the subwavelength resonant frequencies. Unlike the expansion in the previous section, the implicit function theorem is no longer applicable here because zero is a high-order root of the characteristic function $f(k;\sigma)$ in (\ref{eq: f(z; sigma) expansion}) (or equivalently, $g(k; \nu)$ in (\ref{equ:G_expession})). Instead, we employ the Newton polygon method from multivariate complex analysis to obtain the asymptotic expansion (see Appendix \ref{app: Newton Polygon Method}). 

Let 
\begin{equation}\label{equ:defineofx}
    \theta_j=\frac{1}{t_jt_{j+1}}, \quad j=0,1,\cdots,2N-1,
\end{equation}
with $t_0=t_{2N}=1$. We first introduce the expansion of $g_l(k)$ in (\ref{equ:G_expession}) in Proposition \ref{prop:ex_g} below, whose proof is deferred to Appendix \ref{sect:proofexpangl} for ease of exposition. 

\begin{proposition} \label{prop:ex_g}
When $k\to0$, for $1\leq l\leq N-1$, we have
    \begin{equation}\label{equ:g_ex}
    \begin{aligned} 
        g_l(k)=\bigl(\prod_{j=1}^{2N-1}t_j\bigr)\Big[ &  2^l(-2\i k)^{2N-1-2l}\sum_{1\leq j_1\prec j_2\prec \cdots\prec j_l\leq 2N-2}\bigl(\prod_{m=1}^{l}\theta_{j_m}\bigr)  \\ 
        &+\theta_02^{l-1}(-2\i k)^{2N-2l}\sum_{2\leq j_1\prec j_2\prec \cdots\prec j_{l-1}\leq 2N-2}\bigl(\prod_{m=1}^{l}\theta_{j_m}\bigr)  \\ 
        &+\theta_{2N-1}2^{l-1}(-2\i k)^{2N-2l}\sum_{1\leq j_1\prec j_2\prec \cdots\prec j_{l-1}\leq 2N-3}\bigl(\prod_{m=1}^{l}\theta_{j_m}\bigr)\\
        & + O(k^{2N-2l+1})\Big],
    \end{aligned}
    \end{equation}
with $i\prec j$ meaning $j-i>1$, and for $l=0, N$,
    \[
    g_0(k)=\bigl(\prod_{j=1}^{2N-1}t_j\bigr)(-2k\i)^{2N-1}+O(k^{2N+1}), \quad g_N(k)=2^{N}+O(k).
    \]
    
\end{proposition}

Leveraging these asymptotic expansions, we can recover the capacitance matrix theory in \cite{Subwavelength_1D_High-Contrast} from the propagation matrix method. Recall from 
\cite{Subwavelength_1D_High-Contrast} that the generalized capacitance matrix is defined as: with $\theta_j=\tfrac{1}{t_jt_{j+1}}$ in (\ref{equ:defineofx}),
\begin{align}\label{equ: mathcal C def}
\mathcal{C}:=&\begin{pmatrix}
        \theta_1 & -\theta_1\\
        -\theta_2 & \theta_2+\theta_3 & -\theta_3\\
         & -\theta_4 & \theta_4+\theta_5 & -\theta_5\\
         && \ddots & \ddots & \ddots \\
         &&& -\theta_{2N-4} & \theta_{2N-4} + \theta_{2N-3} & -\theta_{2N-3}\\
         &&&& - \theta_{2N-2} & \theta_{2N-2}
    \end{pmatrix}.
\end{align}
We now introduce several auxiliary lemmas.

\begin{lemma}\label{lem:capamatrixspectrum1}
    For the generalized capacitance matrix $\mathcal{C}$, we have
    \begin{enumerate}
        \item It has $N$ different eigenvalues $0=\lambda_1<\lambda_2<\cdots<\lambda_N$. 
        \item Let $P_N$ be the determinant of $\mathcal{C} - \lambda I$, and $Q_N^i$ the $(i, i)$-cofactor for $i = 1, N$. Then  
        \begin{align}
            & P_N(\lambda,\theta_1,\theta_2,\dots \theta_{2N-2}) = |\mathcal{C}-\lambda I| =  \sum_{l=0}^{N-1} (-\lambda)^{N-l}\sum_{{\substack{1\leq j_1\prec j_2\prec \\\cdots\prec j_l\leq 2N-2}}}\prod_{m=1}^{l}\theta_{j_m},  \label{equ: exp_P_N} \\
            &  Q^1_N(\lambda,\theta_2,\theta_3,\dots \theta_{2N-2}) = [\mathcal{C}-\lambda I]_{11} = \sum_{l=1}^{N} (-\lambda)^{N-l}\sum_{{\substack{2\leq j_1\prec j_2\prec \\\cdots\prec j_{l-1}\leq 2N-2}}}\prod_{m=1}^{l-1}\theta_{j_m}, \notag\\
            &  Q^N_N(\lambda,\theta_1,\theta_2,\dots \theta_{2N-3}) = [\mathcal{C}-\lambda I]_{NN} = \sum_{l=1}^{N} (-\lambda)^{N-l}\sum_{{\substack{1\leq j_1\prec j_2\prec \\\cdots\prec j_{l-1}\leq 2N-3}}}\prod_{m=1}^{l-1}\theta_{j_m}\notag.
        \end{align}
        \item   Suppose that $\lambda\neq 0$ is an eigenvalue of $\mathcal{C}$, and $\bm a=(a_1,a_2,\dots,a_N)^{\top}$, $\bm b=(b_1,b_2,\dots,b_N)$ are the corresponding right and left eigenvectors. Then
    \[
    \frac{a_i b_i}{\sum_{j=1}^{N}a_j b_j}=\frac{Q_N^{i}(\lambda)}{P_N'(\lambda)}, \quad i=1,\ N.
    \]
    \end{enumerate}
\end{lemma}

\begin{proof}
    The property (1) follows from the fact that $\mathcal{C}$ is a tridiagonal matrix with non-zero off-diagonal elements; see \cite[Lemma 7.7.1]{parlett1998symmetric}.
    
    For property (2), we prove it by mathematical induction. Firstly, $P_1(\lambda, \theta_1) = -\lambda+\theta_1$. Then by adding the last column of $\mathcal{C} - \lambda I$ to its second-to-last column and expanding along the last column, we obtain the following recurrence formula
    \begin{equation*}
        \begin{aligned}
              P_{N}(\lad;\theta_1,\theta_2,\cdots \theta_{2N-2})=(\theta_{2N-2}-\lad)&P_{N-1}(\lad;\theta_1,\theta_2,\cdots \theta_{2N-4})\\
              &+\theta_{2N-3}P_{N-1}(\lad;\theta_1,\theta_2,\cdots \theta_{2N-5},0).
    \end{aligned}
    \end{equation*}
 Using mathematical induction, we derive the expansion \eqref{equ: exp_P_N} of $P_N(\lambda; \theta_1, \dots, \theta_{2N-2})$. Similarly, the expansions of $Q_N^1$ and $Q_N^N$ can also be obtained.
 
   
    For property (3), since $\lambda$ is a simple eigenvalue of $\mathcal{C}$, the matrix $\mathcal{C} - \lambda I$ has rank $N-1$. Then, the adjugate matrix of $\mathcal{C} - \lambda I$ is a rank-1 matrix, and there exists a $0\ne\beta \in \C$ such that $\operatorname{adj}(\mathcal{C}-\lambda I)=\beta\bm a \bm b$.
    Thus, we have
    \[
    Q^i_N(\lambda)=[\mathcal{C}-\lambda I]_{ii}=\beta a_i b_i,\quad i=1,N, \quad P_N^{'}(\lambda)=\operatorname{tr}\left(\operatorname{adj}(\mathcal{C}-\lambda I)\right)=\beta\sum_{j=1}^{N}a_j b_j.
\]
Combining these equations completes the proof.
\end{proof}

\begin{remark}
Note that we can express the generalized capacitance matrix in \eqref{equ: mathcal C def} as $\mathcal{C} = V^{-1}C$, where $C$ is a symmetric matrix given by
\[
C := \begin{pmatrix}
\frac{1}{t_2} & -\frac{1}{t_2} \\
-\frac{1}{t_2} & \frac{1}{t_2} + \frac{1}{t_4} & -\frac{1}{t_4} \\
& -\frac{1}{t_4} & \frac{1}{t_4} + \frac{1}{t_6} & -\frac{1}{t_6} \\
&& \ddots & \ddots & \ddots \\
&&& -\frac{1}{t_{2N-4}} & \frac{1}{t_{2N-4}} + \frac{1}{t_{2N-2}} & -\frac{1}{t_{2N-2}} \\
&&&& -\frac{1}{t_{2N-2}} & \frac{1}{t_{2N-2}}
\end{pmatrix},
\]
and $V := \operatorname{diag}(t_1, t_3, \cdots, t_{2N-1})$. This decomposition implies a direct relation between the right and left eigenvectors. Specifically, if $\bm{a} = (a_1, a_2, \dots, a_N)^{\top}$ and $\bm{b} = (b_1, b_2, \dots, b_N)$ are the right and left eigenvectors of $\mathcal{C}$ associated with the same eigenvalue, respectively, then there holds 
\begin{align*}
    \bm{b}^\top = V \bm{a},\q \text{equivalently,}\q b_j = t_{2j-1} a_j, \quad \text{for } j = 1, 2, \cdots, N.
\end{align*}


\end{remark}

We now present the asymptotic expansions of the scattering resonance $\omega(\delta)$ in the subwavelength regime. In what follows, \(\sqrt{z}\) is defined on \(\C\setminus[0,+\infty)\) with the branch \(\im \sqrt{z}>0\). For $z\in[0,+\infty)$, we define $\sqrt{z}:=\lim_{\varepsilon\to0_+}\sqrt{z+\varepsilon\i}$. We remark that the following result can be generalized to the case where the $v_b$'s differ across the $D_i$'s.

\begin{theorem}\label{thm: subwavelength_resonant_frequencies}
Let $\delta \in \mathbb{C}$. In the cases where $\delta \to 0$ (or $\delta \to \infty$), the scattering problem \eqref{equ: scattering problem} exhibits exactly $2N$ subwavelength resonant frequencies:
    \begin{itemize}
        \item[$\bullet$] a trivial frequency $\omega^*=0$.
        \item[$\bullet$] the first nontrivial eigenfrequency $\ww_1(\delta)$ is an analytic function of $\delta$ (or $\delta^{-1}$), with its leading asymptotic expansion given by 
        \begin{align*}
        &\ww_1(\d)=-2\i \d\frac{v}{r^2\sum_{j=1}^N\ell_{j}}+O(\d^2),\quad\d\to0,\\
        \text{or\quad}&\ww_1(\delta)=-2\i\frac{1}{\d}\frac{v}{\sum_{j=1}^N\ell_j}+O(\d^{-2}),\quad \d\to\infty.
        \end{align*}
        \item[$\bullet$] the remaining \(2N-2\) frequencies are analytic functions of \(\delta^{\frac{1}{2}}\) (or \(\delta^{-\frac{1}{2}}\)), and their leading-order asymptotic expansion is given by 
        \begin{align*}
        &\ww_i^{\pm}(\d)=\pm v\sqrt{\frac{\lad_i\d}{r}}-\i \d\frac{v}{2r^2}\frac{a_{i1}^2+a_{iN}^2}{\sum_{j=1}^{N}a_{ij}^2\ell_{j}}+O(\d^{\frac{3}{2}}), \ \d\to0,\q 2\leq i\leq N,\\
        \text{or\quad}&\ww_i^{\pm}(\d)=\pm v\sqrt{\frac{\lambda_ir}{\d}}-\i \frac{1}{\delta}\frac{v}{2}\frac{a_{i1}^2+a_{iN}^2}{\sum_{j=1}^{N}a_{ij}^2\ell_{j}}+O(\d^{-\frac{3}{2}}),\ \d\to \infty,\q 2\leq i\leq N,
        \end{align*}
        where $0=\lambda_1<\lambda_2<\cdots<\lambda_N$ are eigenvalues of the matrix $\mathcal{C}$ \eqref{equ: mathcal C def} and $\bm a_i=(a_{i1},a_{i2},\cdots a_{iN})^{\top}$ is the associated eigenvectors.
    \end{itemize}
\end{theorem}

\begin{proof}
Similar to the proof of Theorem \ref{thm:resonancesexpan1}, it suffices to consider the case \(\delta \to 0\) and derive the asymptotics of the wavenumber \(k(\nu)\) in terms of \(\nu = \tfrac{2\delta/r}{1+\delta/r}\).

First, we apply the Newton polygon method reviewed in Appendix \ref{app: Newton Polygon Method} to determine the leading-order asymptotics of \(k(\nu)\). Specifically, by \eqref{equ:G_expession}, we expand \(g(k; \nu)\) around \(k = 0\) as:
\begin{equation*}
g(k;\nu)=\sum_{l=1}^{2N}g_l(k)\nu^l=\sum_{l=1}^{2N}\sum_{j=2N-1-2l}^\infty c_{jl}k^j\nu^{l}. 
\end{equation*}
Here, \(c_{jl}\) can be determined by \(g_l(k)\) in Proposition \ref{prop:ex_g}. The associated Newton polygon can be plotted in the \(\mathbb{R}^2\) plane, where the lower boundary of the convex hull comprises two connected piecewise linear segments, \(l_1\) and \(l_2\), with slopes \(-1\) and \(-\tfrac{1}{2}\), respectively, as shown in Fig. \ref{fig: lower boundary}.

    \begin{figure}[htbp!]
    \centering
    \begin{tikzpicture}[scale=0.8]
        \def\N{4}
        
        \pgfmathsetmacro{\xmax}{2*\N-1}
        \pgfmathsetmacro{\yA}{\N}           
        \pgfmathsetmacro{\yB}{\N-1}         
        \pgfmathsetmacro{\yC}{0}            
        
        \draw[->, thick] (0,0) -- (0,\N+1) node[left] {degree of $\nu$};
        \draw[->, thick] (0,0) -- (\xmax+4,0) node[below] {degree of $k$};
        
        \node at (0,0) [below left] {$O$};
        
        \draw[ultra thick, blue] 
            (0,\yA) -- 
            node[pos=0.7, above right] {$l_1$} 
            (1,\yB);
        
        \draw[ultra thick, blue] 
            (1,\yB) -- 
            node[pos=0.5, above right] {$l_2$} 
            (\xmax,\yC);

        
        \fill[red] (0,\yA) circle (2pt);
        \fill[red] (1,\yB) circle (2pt);
        \fill[red] (\xmax,\yC) circle (2pt);
        
        \draw[dashed, gray] (1,\yB) -- (1,0);  
        \draw[dashed, gray] (1,\yB) -- (0,\yB); 
        
        \draw (0.1,\yA) -- (-0.1,\yA) node[left] {$N$};
        \draw (0.1,\yB) -- (-0.1,\yB) node[left] {$N-1$};
        
        \draw (1,0.1) -- (1,-0.1) node[below] {$1$};
        \draw (\xmax,0.1) -- (\xmax,-0.1) node[below] {$2N-1$};
        
        \fill[black] (1,0) circle (1.5pt);  
        \fill[black] (0,\yB) circle (1.5pt); 
    \end{tikzpicture}
    
    \caption{The lower boundary of the convex hull}
    \label{fig: lower boundary}
    \end{figure}
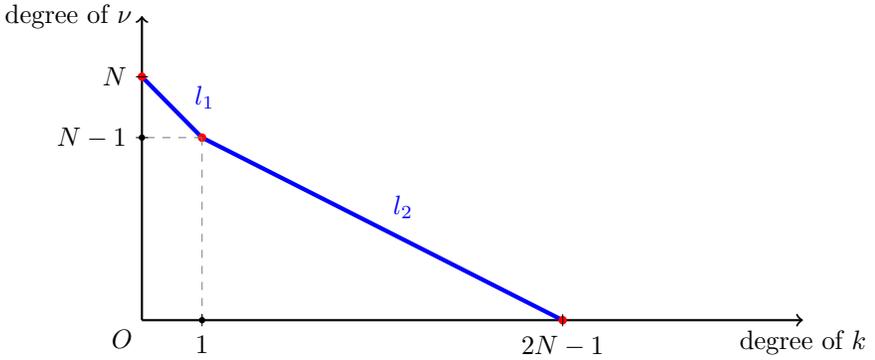

 The sets of points on the edges $l_1$ and $l_2$ are defined as follows:  
 \begin{align*}
    & E_1 := l_1 \cap \mathbb{Z}^2 = \{(0, N), (1, N-1)\}, \\
    & E_2 := l_2 \cap \mathbb{Z}^2 = \{(1, N-1), (3, N-2), (5, N-3), \dots, (2N-1, 0)\}.
 \end{align*}
 For the first edge \(l_1\), by \eqref{eq:slope}, the asymptotic form is \(k \sim c\nu\) with \(c \neq 0\), and the balance equation is  
 \begin{equation*}
     \sum_{(j, l) \in E_1} c_{jl} c^j = 0,
 \end{equation*}
where, using \eqref{equ:g_ex}, we have  
\begin{equation*}
    c_{0N} = 2^N, \quad c_{1, N-1} = -2^N \i \sum_{j=1}^N t_{2j-1}.
\end{equation*}
  Recall that \(l_j = t_{2j-1}\) for \(j = 1, \dots, N\). Solving the balance equation gives \(c^{-1} = \i \sum_{j=1}^N t_{2j-1}\), leading to the asymptotic form of the zero:  
\begin{equation*}
    k_1(\nu) \sim -\i \frac{\nu}{\sum_{j=1}^N t_{2j-1}}, \quad \nu \to 0.
\end{equation*}
Similarly, for the second edge \(l_2\), the asymptotic form is \(k \sim c\sqrt{\nu}\) with \(c \neq 0\), and the balance equation is given by  
\begin{align}\label{equ: balance equation E2}  
\sum_{(j, l) \in E_2} c_{jl} c^j = 0,  
\end{align}  
where, according to \eqref{equ:g_ex},  
\begin{align*}
    c_{2(N-l)-1, l} =  2^l (-2\i)^{2(N-l)-1} \prod_{j=1}^{2N-1} t_j\sum_{\substack{1 \leq j_1 \prec j_2 \prec\\ \cdots \prec j_l \leq 2N-1}} \prod_{m=1}^l \theta_{j_m}, \quad 0 \leq l \leq N-1.
\end{align*}
Using the the characteristic polynomial \(P_N(z)\) from Lemma \ref{lem:capamatrixspectrum1}, we have  
  \begin{equation*}
      0 = \sum_{(j, l) \in E_2} c_{jl} c^j = \i 2^{N-1} \bigl(\prod_{j=1}^{2N-1} t_j\bigr)\frac{P_N(2c^2)}{c}. 
  \end{equation*}
It follows that the solutions to the equation \eqref{equ: balance equation E2} can be characterized by the nonzero eigenvalues $\lad_j$, $ 2 \leq j \leq N$, of \(\mathcal{C}\):  
\begin{equation*}
    c_i^\pm = \pm \sqrt{\frac{\lambda_i}{2}}, \quad 2 \leq i \leq N.
\end{equation*}
    Consequently, the other $2N-2$ zeros have the asymptotic form
    \[
    k_i^\pm(\nu)\sim\pm\sqrt{\frac{1}{2}\lambda_i\nu},\quad\nu\to0,\quad2\leq i\leq N.
    \]
    
Next, we compute the second-order correction terms. Define 
    \[
    h_i^{\pm}(\alpha;\nu^{\frac{1}{2}}):=\nu^{-N+\frac{1}{2}}g\bigl(\nu^{\frac{1}{2}}\bigl(\pm\sqrt{\frac{1}{2}\lambda_i}+\alpha\bigr),\nu\bigr),\quad 2\leq i \leq N,
    \]
    with $g$ defined in (\ref{equ:Non_Matrix rep3}). A direct computation gives 
    \begin{align} 
       h_i^{\pm}&(\alpha;\nu^{\frac{1}{2}})\label{eq:h_alpha} \\
        = & \nu^{\frac{1}{2}}\sum_{l=1}^{N}c_{2N-2l,l}\bigl(\pm\sqrt{\frac{1}{2}\lambda_i}\bigr)^{2N-2l} \notag  + \alpha\sum_{l=0}^{N-1}(2N-2l-1)c_{2N-2l-1,l}\bigl(\pm\sqrt{\frac{1}{2}\lambda_i}\bigr)^{2N-2l-2} \\&+\sum_{\substack{j,l\in \Z\\2j+l\ge 2}}d_{jl}\nu^{j}\alpha^{l} \notag \\
        =&\bigl(\prod_{j=1}^{2N-1} t_j\bigr)\left[\nu^{\frac{1}{2}}2^{N-1}\bigl(\theta_0Q_N^1(\lambda_i)\!+\!\theta_{2N-1}Q_N^N(\lambda_i)\bigr)\!-\!\alpha \i 2^{N+1}P_N^{'}(\lambda_i)\right]\!+\!\sum_{\substack{j,l\in \Z\\2j+l\ge 2}}d^{\pm}_{jl}\nu^{j}\alpha^{l}\notag. 
    \end{align}
   Notice that \(h_i^\pm(\alpha; \nu^{\frac{1}{2}})\) is  analytic in \(\alpha\) and \(\nu^{\frac{1}{2}}\), with $h_i^\pm(0; 0) = 0$, and satisfies 
\begin{align*}
  \frac{\partial}{\partial \alpha} h_i^\pm(0; 0) &= - 2^{N+1}\i \bigl(\prod_{j=1}^{2N-1} t_j\bigr)P_N'(\lambda_i),\\
  \frac{\p }{\p \nu^{\frac{1}{2}}}h_i^\pm(0; 0)&=2^{N-1}\bigl(\prod_{j=1}^{2N-1} t_j\bigr)\bigl(\theta_0Q_N^1(\lambda_i)+\theta_{2N-1}Q_N^N(\lambda_i)\bigr).  
\end{align*}
Since the eigenvalue \(\lambda_i\) has multiplicity one for $\mathcal{C}$, it follows that \(P_N'(\lambda_i) \neq 0\). By the implicit function theorem, there exists a unique analytic function \(\alpha_i^\pm(\nu^{\frac{1}{2}})\) such that  
\[
h_i^\pm\left(\alpha_i^\pm(\nu^{\frac{1}{2}}); \nu^{\frac{1}{2}}\right) = 0, \quad 2 \leq i \leq N,
\]  
with \(\alpha_i^\pm(0) = 0\), for \(\nu\) sufficiently close to zero.
Moreover, using the implicit function
theorem method for the expansion \eqref{eq:h_alpha}, and combining Lemma \ref{lem:capamatrixspectrum1}, we have
    \[
    \alpha_i^{\pm}(\nu^{\frac{1}{2}})=-\nu^{\frac{1}{2}}\frac{\i}{4}\frac{a_{i1}^2+a_{iN}^2}{\sum_{j=1}^{N}a_{ij}^2t_{2j-1}}+O(\nu),\quad 2\leq i \leq n.
    \]
 Thus, \(k_i^\pm(\nu) = \nu^{\tfrac{1}{2}} \left(\pm \sqrt{\frac{1}{2} \lambda_i} + \alpha_i^\pm(\nu^{\frac{1}{2}})\right), \; 2 \leq i \leq N,\) are the \(2N-2\) zeros of \(g(k; \nu)\), which are analytic functions of \(\nu^{\frac{1}{2}}\). Their leading-order asymptotic expansions are given by  
 \begin{equation*}
     k_i^\pm(\nu) = \pm \sqrt{\frac{\lambda_i \nu}{2}} - \i \frac{\nu}{4} \frac{a_{i1}^2 + a_{iN}^2}{\sum_{j=1}^N a_{ij}^2 t_{2j-1}} + O(\nu^{\frac{3}{2}}), \quad i = 2, 3, \dots, N.
 \end{equation*}
 Finally, for \(k_1(\nu)\), define  
  \begin{equation*}
      h_1(\alpha, \nu) = \nu^{-N} g\left(\nu \left(\i \frac{\nu}{\sum_{j=1}^N t_{2j-1}} + \alpha\right); \nu\right).
  \end{equation*}
We find that \(h_1(0, 0) = 0\) and \(\tfrac{\partial}{\partial \alpha} h_1(0, 0) = c_{1, N-1} \neq 0\). By the implicit function theorem, it follows that \(k_1(\nu)\) is an analytic function of \(\nu\).
\end{proof}

\subsection{Eigenmodes of the scattering problem}

In this subsection, we characterize the eigenmodes of \eqref{equ: scattering problem} corresponding to the $2N-1$ nontrivial subwavelength resonant frequencies. The following Theorem \ref{thm:subwaveeigenmodes1} treats the case $\d\to0$; the result for $\d\to \infty$ follows from Theorem \ref{thm: resonant frequency property} (see Corollary \ref{coro:infite}). 


\begin{theorem} \label{thm:subwaveeigenmodes1}
    Suppose that $\ww(\d) =k(\d)v=O(\d^{1/2})$ is a subwavelength resonant frequency for the scattering problem \eqref{equ: scattering problem} in case $\d\to0$ and analytic about $\d^{1/2}$, $u(x)$ is the corresponding non-trivial solution with $u(x)=\e^{-\i\frac{\ww}{v}x}$ for $x<x_1^{-}$, then
    \begin{equation}\label{equ: u(x) estimate}
    \begin{split}
        &u(x)=\begin{cases}
            a_j+O(\d^{1/2}) &x\in (x_j^{-}, x_j^{+}),\qquad j=1,2,\cdots ,N,\\
            a_j+b_j(x-x_j^{+})+O(\d^{1/2}) &x\in(x_j^{+},x_{j+1}^{-})\qquad j=0,1,2,\cdots ,N,
        \end{cases}\\
        &u'(x)=\begin{cases}
            O(\d) &x\in (x_j^{-}, x_j^{+}),\qquad j=1,2,\cdots N,\\
            b_j+O(\d^{1/2}) &x\in(x_j^{+},x_{j+1}^{-})\qquad j=0,1,2,\cdots ,N,
        \end{cases}
    \end{split}
    \end{equation}
    where $\bm a=(a_1,a_2,\cdots,a_N)^{\top}$ is the corresponding eigenvector of the capacitance matrix $\mathcal{C}$ and $b_j=\tfrac{a_{j+1}-a_j}{s_j},j=1,2,\cdots,N-1$. Here, $x_0^{+}:= - M$ and $x_{N+1}^{-}:= M$ for a large enough $M$. 
\end{theorem}
\begin{proof}
The proof proceeds in two steps. In Step 1, we derive the expressions for $u(x)$ and $u'(x)$ in \eqref{equ: u(x) estimate} by mathematical induction. In Step 2, we determine the constants $a_j$ and $b_j$.

\textbf{Step 1.}  For $j=0$, on interval $(x_0^+,x_1^-)$, $u(x)=\e^{-\i\frac{\ww}{v}x}=1+O(\delta^{1/2}),u'(x)=O(\d^{1/2})$, so (\ref{equ: u(x) estimate}) holds for the case when $j=0$. Now, suppose that for $x\in(x_j^{+},x_{j+1}^{-})$,
    \[
    \begin{cases}
        u(x)=a_j+b_j(x-x_j^{+})+O(\d^{1/2}), \\
        u'(x)=b_j+O(\d^{1/2}). 
    \end{cases}
    \]
It follows that
    \[
    \begin{cases}
        u(x_{j+1}^{-})=a_j+b_j(x_{j+1}^{-}-x_j^{+})+O(\d^{1/2}):=a_{j+1}+O(\d^{1/2}),\\
        \left.\frac{\dd u}{\dd x}\right\vert_{+}(x_{j+1}^{-})=\d\left.\frac{\dd u}{\dd x}\right\vert_{-}(x_{j+1}^{-})=O(\d).
    \end{cases}
    \]
    Furthermore, from the propagation matrix, we have 
    \begin{align*}
        u(x)&=\cos\left(kr\left(x-x_{j+1}^{-}\right)\right)u|_+(x_{j+1}^{-})+\frac{1}{kr}\sin\left(kr\left(x-x_{j+1}^{-}\right)\right)u'|_+(x_{j+1}^{-})\\
        &=(1+O(\d))u|_+(x_{j+1}^-)+[(x-x_{j+1}^{-})+O(\d)]u'|_+(x_{j+1}^{-})\\
        &=a_{j+1}+O(\d^{1/2}),\\
        u'(x)&=-kr\sin\left(kr\left(x-x_{j+1}^{-}\right)\right)u|_+(x_{j+1}^{-})+\cos\left(kr\left(x-x_{j+1}^{-}\right)\right)u'|_+(x_{j+1}^{-})\\
        &=O(\d)u|_+(x_{j+1}^{-})+(1+O(\d))u'|_+(x_{j+1}^{-})\\
        &=O(\d),
    \end{align*}
    for $x\in(x_{j+1}^-,x_{j+1}^+)$. At the end point $x_{j+1}$, the jump relation gives 
    \[
        \begin{cases}
        u|_+(x_{j+1}^{+})=a_{j+1}+O(\d^{1/2}),\\
        \left.\frac{\dd u}{\dd x}\right\vert_{+}(x_{j+1}^{+})=\frac{1}{\d}\left.\frac{\dd u}{\dd x}\right\vert_{-}(x_{j+1}^{+})=O(1)=b_{j+1}+O(\d^{1/2}).
    \end{cases}
    \]
Then, for $x\in(x_{j+1}^+,x_{j+2}^-)$, it holds that 
    \begin{align*}
        u(x)&=\cos\left(k\left(x-x_{j+1}^{+}\right)\right)u|_+(x_{j+1}^{+})+\frac{1}{k}\sin\left(k\left(x-x_{j+1}^{+}\right)\right)u'|_+(x_{j+1}^{+})\\
        &=(1+O(\d))u|_+(x_{j+1}^+)+[(x-x_{j+1}^{+})+O(\d)]u'|_+(x_{j+1}^{+})\\
        &=a_j+b_{j+1}(x-x_{j+1}^{+})+O(\d^{1/2}),
          \end{align*}
and 
    \begin{align*}
        u'(x)&=-kr\sin\left(kr\left(x-x_{j+1}^{+}\right)\right)u|_+(x_{j+1}^{+})+\cos\left(kr\left(x-x_{j+1}^{+}\right)\right)u'|_+(x_{j+1}^{+})\\
        &=O(\d)u|_+(x_{j+1}^{+})+(1+O(\d))u'|_+(x_{j+1}^{+})\\
        &=b_{j+1}+O(\d^{1/2}).
    \end{align*}
This proves \eqref{equ: u(x) estimate} by mathematical induction.

\textbf{Step 2.} Using the propagation matrix, we have
    \begin{align*}
        \begin{pmatrix}
        u(x_{j+1}^-)\\u'(x_{j+1}^-)
    \end{pmatrix}&=
    \begin{pmatrix}
        \cos(ks_j) & \frac{1}{k}\sin(ks_j)\\
        -k\sin(ks_j) & \cos(ks_j)
    \end{pmatrix}
    \begin{pmatrix}
        \cos(krl_j) & \frac{\d}{kr}\sin(krl_j)\\
        -\frac{kr}{\d}\sin(krl_j) & \cos(krl_j)
    \end{pmatrix}\begin{pmatrix}
        u(x_{j}^-)\\u'(x_{j}^-)
    \end{pmatrix}\\
    &=\left[\begin{pmatrix}
        1-rl_js_j\lambda & s_j\\
        -rl_j\lambda & 1
\end{pmatrix}+O(\d^{1/2})\right]\begin{pmatrix}
        u(x_{j}^-)\\u'(x_{j}^-)
    \end{pmatrix},\quad j=1,2,\cdots,N-1,
    \end{align*}
    where $u'(x)$ is understand as $\left.\frac{\dd u}{\dd x}\right\vert_{-}$. By the formula (\ref{equ: u(x) estimate}), we also have
    \[
    \begin{pmatrix}
        u(x_{j}^-)\\u'(x_{j}^-)
    \end{pmatrix}=\begin{pmatrix}
        a_{j}\\ \frac{a_{j}-a_{j-1}}{s_{j-1}}
    \end{pmatrix}+O(\d^{1/2}),\quad j=2,\cdots,N,\qquad \begin{pmatrix}
        u(x_{1}^-)\\u'(x_{1}^-)
    \end{pmatrix}=\begin{pmatrix}
        a_{1}\\ 0
    \end{pmatrix}+O(\d^{1/2}).
    \]
    Combining the two relations, we can obtain that 
    \begin{align*}
    &\frac{a_{j+1}}{rl_js_j}+\frac{a_{j-1}}{rl_js_{j-1}}=\left(\frac{1}{rl_js_{j-1}}+\frac{1}{rl_js_j}-\lambda\right)a_j,\quad j=2,3,\cdots,N-1,\\
    &\frac{1}{rl_1s_1}a_2=\left(\frac{1}{rl_1s_1}-\lambda\right)a_1.
    \end{align*}
If we imagine adding an interval of length $s_N$ after the last resonator, we can use the propagation matrix to obtain
\begin{equation*}
     \begin{pmatrix}
        a_N\\0
    \end{pmatrix}=\begin{pmatrix}
        1-rl_Ns_N\lambda&s_N\\
        -rl_N\lambda&1
    \end{pmatrix}\begin{pmatrix}
        a_N\\ \frac{a_N-a_{N-1}}{s_{N-1}}
    \end{pmatrix}.
\end{equation*}
This yields
\begin{align*}
     \frac{1}{rl_{N}s_{N-1}}a_{N-1}=\left(\frac{1}{rl_{N}s_{N-1}}-\lambda\right)a_N. 
\end{align*}
    Combining these equations, we obtain $(\mathcal{C}-\lambda I)\bm a=\bm 0$, which completes the proof.
\end{proof}

In the case $\d\to\infty$, a similar characterization follows from Theorem \ref{thm: resonant frequency property}.

\begin{corollary} \label{coro:infite}
    Under the assumption of Theorem \ref{thm:subwaveeigenmodes1}, we choose $t = \delta_0^{1/2}$ in Theorem \ref{thm: resonant frequency property}(3), such that for $x < x_1^- $, $v(x) = \delta_0^{1/2}u(x)$ is the corresponding eigenmode for $\delta = r^2/\delta_0$. Then, as $\delta_0 \to 0$ (namely, $\delta \to \infty$), we have
    \begin{equation}\label{equ: v(x) estimate}
    \begin{split}
        &v(x)=\begin{cases}
            c_j+O(\d^{-1/2}) &x\in(x_{j-1}^{+},x_{j}^{-})\qquad j=1,2,\cdots ,N+1,\\
            c_j+a_j(x-x_j^{+})+O(\d^{-1/2}) &x\in (x_j^{-}, x_j^{+}),\qquad j=1,2,\cdots ,N.\\
        \end{cases}
    \end{split}
    \end{equation}
Here $\bm a=(a_1,a_2,\ldots,a_N)^{\top}$ is the corresponding eigenvector of the capacitance matrix $\mathcal{C}$, and $c_1=0$, $c_j=c_{j-1}+a_{j-1}\ell_{j-1}$ for $j=2,\ldots,N$.
\end{corollary}

\subsection{Comparison with the three-dimensional case as \texorpdfstring{$\delta\to\infty$}{δ→∞}}\label{sec:threeddeltainfinity1}

In this subsection, we present a concise comparison between the subwavelength resonant frequencies in the one-dimensional and three-dimensional settings. To this end, we consider a three-dimensional analogue of the scattering problem~\eqref{equ: scattering problem}; see~\cite{ammari2018book}. In the three-dimensional case, assuming that $k^2$ is not a Dirichlet eigenvalue of $-\Delta$ on $D$, the scattering problem admits a resonant frequency if and only if the operator $\mathcal{A}(k,\delta)$ is not injective. Here, $\mathcal{A}(k,\delta): L^2(\partial D) \times L^2(\partial D) \to H^1(\partial D) \times L^2(\partial D)$ is defined by
\begin{align*}
\mathcal{A}(k,\delta) = \begin{pmatrix}
    \mathcal{S}_D^{rk} & -\mathcal{S}_D^k \\
    -\frac{1}{2}I + \mathcal{K}_D^{rk,*} & -\delta\left(\frac{1}{2}I + \mathcal{K}_D^{k,*}\right)
\end{pmatrix},
\end{align*}
where $\mathcal{S}^\ww_D$ is the single layer potential and $\mathcal{K}^{\ww,*}_D$ represents the Neumann-Poincar\'e operator. Introducing the parameter $\tau = \delta^{-1}$, we observe that $\mathcal{A}(k,\delta)$ is injective if and only if $\widetilde{\mathcal{A}}(k,\tau)$ is injective, where
\[
\widetilde{\mathcal{A}}(k,\tau) = \begin{pmatrix}
    \mathcal{S}_D^{rk} & -\mathcal{S}_D^k \\
    -\tau\left(\frac{1}{2}I - \mathcal{K}_D^{rk,*}\right) & -\left(\frac{1}{2}I + \mathcal{K}_D^{k,*}\right)
\end{pmatrix}.
\]
Consequently, the subwavelength resonance problem in the limit $\delta \to \infty$ reduces to finding $\ww(\tau)$ such that $\ww(\tau) \to 0$ as $\tau \to 0$ and $\widetilde{\mathcal{A}}(k, \tau)$ is not injective. It is straightforward to verify that $\widetilde{\mathcal{A}}(k,\tau)$ is continuous in the norm topology. Examining the operator $\widetilde{\mathcal{A}}(0,0)$, we find that
\[
\widetilde{\mathcal{A}}(0,0) = \begin{pmatrix}
    \mathcal{S}_D^0 & -\mathcal{S}_D^0 \\
    0 & -\left(\frac{1}{2}I + \mathcal{K}_D^{0,*}\right)
\end{pmatrix}
\]
is invertible, which follows from the invertibility of $\mathcal{S}_D^0: L^2(\partial D) \to H^1(\partial D)$ and the fact that the spectrum of $\mathcal{K}_D^{0,*}$ is contained in $(-1/2, 1/2]$; see again~\cite{ammari2018book}. Since the set of invertible operators forms an open subset in the space of bounded operators equipped with the norm topology, $\widetilde{\mathcal{A}}(k,\tau)$ remains invertible in a neighborhood of $(k,\tau) = (0,0)$. This implies that the three-dimensional scattering problem possesses no subwavelength resonant frequencies in the limit $\delta \to \infty$. This behavior stands in stark contrast to the one-dimensional case, where Theorem~\ref{thm: subwavelength_resonant_frequencies} establishes the existence of exactly $2N$ subwavelength resonant frequencies as $\delta \to \infty$.

\section{Non-reciprocal system}\label{sec:nonresys}

In this section, we generalize the propagation matrix method to wave propagation in systems of non-Hermitian high-contrast resonators. We consider the same chain of resonators as in Section \ref{sec:modelsethermitian}. By introducing an imaginary gauge potential $\gamma$ (see \cite{ammari2024mathematical, yokomizo2022non}), the governing equation for wave propagation becomes a generalized Sturm--Liouville equation: for $x \in \R$, 
\begin{align}
    -\frac{\omega^{2}}{\kappa(x)}u(x)- \gamma(x) \frac{\dd}{\dd x}u(x)-\frac{\dd}{\dd x}\left( \frac{1}{\rho(x)}\frac{\dd}{\dd
    x}  u(x)\right) =0, 
    \label{eq: gen Strum-Liouville}
\end{align}
with $\gamma(x) = \gamma_i$ for $x \in D_i$, and $0$ for $x\in \R \setminus D$. 
The other parameters are the same as in \eqref{equ: k(x) def} and \eqref{eq:contrast}. In this setting, the wave problem  \eqref{eq: gen Strum-Liouville} can be rewritten as the following system of coupled one-dimensional equations:
\begin{align}
    \begin{dcases}
    \frac{{\dd}^{2}}{\dd x^2}u(x) + \g \frac{\dd}{\dd x}u(x) +\frac{\omega^2}{v_b^2}u(x)=0, & x\in D_i,\\
    \frac{{\dd}^{2}}{\dd x^2}u(x)  +\frac{\omega^2}{v^2}u(x)=0, & x\in \R\setminus D,\\
        u\vert_{+}(x_i^{\pm}) = u\vert_{-}(x_i^{\pm}), & \text{for all } 1\leq i\leq N,\\
        \left.\frac{\dd u}{\dd x}\right\vert_{\pm}(x_i^{\mp})=\delta\left.\frac{\dd u}{\dd x}\right\vert_{\mp}(x_i^{\mp}), & \text{for all } 1\leq i\leq N,\\
        \frac{\dd u}{\dd\abs{x}}(x) -\i \frac{\omega}{v} u(x) = 0, & x\in(-\infty,x_1^{-}) \cup (x_{N}^{+},\infty).
    \end{dcases}
\label{eq:coupled ods}
\end{align}

\subsection{Propagation matrix}
We now formulate the propagation matrix for the system (\ref{eq:coupled ods}). Consider the second-order ODE:  
\begin{align} \label{simple}  
\frac{\dd^2}{\dd x^2}u(x) + \gamma \frac{\dd}{\dd x}u(x) + k^2 u(x) = 0, \quad x \in (0, a),  
\end{align}  
which can be reformulated as the first-order ODE system \(U'(x) = AU(x)\), where  
\[
A := \begin{pmatrix}  
0 & 1 \\  
-k^2 & -\gamma  
\end{pmatrix}, \quad  
U(x) := \begin{pmatrix}  
u(x) \\  
u'(x)  
\end{pmatrix}.  
\]
Its solution is given by, for $x \in [0,a]$, 
\begin{align}\label{equ:odesolution2}
    U(x)=\e^{x A}U(0).
\end{align}
The matrix \(A\) has two eigenvalues given by  
\[
\lambda_1 = \frac{-\gamma - \sqrt{\Delta}}{2}, \quad \lambda_2 = \frac{-\gamma + \sqrt{\Delta}}{2}, \quad \text{where } \Delta = \gamma^2 - 4k^2.
\]  
It can be diagonalized as  
\[
A = P \begin{pmatrix}  
\lambda_1 & 0 \\  
0 & \lambda_2  
\end{pmatrix} P^{-1}, \q P = \begin{pmatrix}  
1 & 1 \\  
\lambda_1 & \lambda_2  
\end{pmatrix}.
\]  
Therefore, we have the propagation matrix from $x = 0$ to $x = a$: 
\begin{align*}
    \e^{aA}=&P\begin{pmatrix}
        \e^{\lad_1a}&0\\0&\e^{\lad_2a}
    \end{pmatrix}P^{-1}=\frac{1}{\lad_2-\lad_1}\begin{pmatrix}
        \lad_2\e^{\lad_1a}-\lad_1\e^{\lad_2a}&\e^{\lad_2a}-\e^{\lad_1a}\\
        -\lad_1\lad_2(\e^{\lad_2a}-\e^{\lad_1a})&\lad_2\e^{\lad_2a}-\lad_1\e^{\lad_1a}
    \end{pmatrix}\\
    =&   \e^{-\frac{\g a}{2}}M(k) T(k;\g,a) M\left(\frac{1}{k}\right),
\end{align*}
where $M(\cdot)$ is given as in \eqref{eq:deftka}, and $T(\cdot)$ is defined by 
\begin{align*}
 T(k; \gamma, a) := \begin{pmatrix}  
\frac{\gamma}{\sqrt{\Delta}} \sinh\left(\frac{\sqrt{\Delta}a}{2}\right) + \cosh\left(\frac{\sqrt{\Delta}a}{2}\right) & \frac{2k}{\sqrt{\Delta}} \sinh\left(\frac{\sqrt{\Delta}a}{2}\right) \\  
-\frac{2k}{\sqrt{\Delta}} \sinh\left(\frac{\sqrt{\Delta}a}{2}\right) & -\frac{\gamma}{\sqrt{\Delta}} \sinh\left(\frac{\sqrt{\Delta}a}{2}\right) + \cosh\left(\frac{\sqrt{\Delta}a}{2}\right)  
\end{pmatrix}.
\end{align*}
Here, the factor \(\e^{-\frac{\gamma a}{2}}\) describes the attenuation of the wave during its propagation (see \cite{ammari2025competingedgebulklocalisation}).  Furthermore, similar to the derivation of (\ref{eq:Matrix rep2}), we find that \(\omega\) is a non-trivial resonant frequency of (\ref{eq:coupled ods}) if and only if \(k = \ww/v\) satisfies 
\begin{multline} \label{prob1non}
      \begin{pmatrix}
        t \\\i t
    \end{pmatrix}=\prod_{j=1}^N \e^{-\frac{\g_{j} a}{2}} M\left(\frac{r}{\d}\right)T(rk; \g_N,\ell_N)M\left(\frac{\d}{r}\right)T(k;0,s_{N-1})M\left(\frac{r}{\d}\right) \\ \cdots T(rk;\g_{1},\ell_{1})M\left(\frac{\d}{r}\right)\begin{pmatrix}
        s \\-\i s
    \end{pmatrix}
\end{multline}
for some $s, t\neq 0$. 
Note $T(kr;\gamma,l)=\tfrac{1}{r}T(k;\tfrac{\gamma}{r},rl)$, and let
\begin{align*}
    \bm t = (r\ell_1,s_{1},r\ell_{2}, \cdots ,s_{N-1}, r\ell_{N}) \in \R_{>0}^{2N-1}\,, \q \bm \beta=(\tfrac{\g_1}{r},0, \tfrac{\g_{2}}{r}, \cdots ,0,\tfrac{\g_N}{r}),\quad  \sigma = \tfrac{\delta}{r}. 
\end{align*}
Then the equation \eqref{prob1non} reduces to finding all \(k \in \C\) satisfying 
\begin{equation} \label{equ:Non_Matrix rep_1}  \begin{aligned}
    t \begin{pmatrix}  
        1 \\ \i  
    \end{pmatrix} = M\left(\frac{1}{\sigma}\right) T(k; \beta_{2N-1}, t_{2N-1}) M(\sigma) T(k; \beta_{2N-2}, t_{2N-2}) M\left(\frac{1}{\sigma}\right)\\ \cdots M\left(\frac{1}{\sigma}\right) T(k; \beta_1, t_1) M(\sigma) \begin{pmatrix}  
        1 \\ -\i  
    \end{pmatrix}.  
    \end{aligned}
\end{equation}

Let \(e_{+} = (1, \i)^T\) and \(e_{-} = (1, -\i)^T\). We observe that  
\begin{align*}  
    &T(k; \beta_i, t_i)e_{\pm} = \left(\pm \frac{2k\i}{\sqrt{\Delta_i}} \sinh\bigl(\frac{\sqrt{\Delta_i}t_i}{2}\bigr) + \cosh\bigl(\frac{\sqrt{\Delta_i}t_i}{2}\bigr)\right)e_{\pm} + \frac{\beta_i}{\sqrt{\Delta_i}} \sinh\bigl(\frac{\sqrt{\Delta_i}t_i}{2}\bigr)e_{\mp}; \\  
    &M(\sigma)e_{\pm} = \frac{1 + \sigma}{2}e_{\pm} + \frac{1 - \sigma}{2}e_{\mp}.  
\end{align*}
Therefore, under the basis \(\{e_{+}, e_{-}\}\), (\ref{equ:Non_Matrix rep_1}) can be rewritten as  
\begin{align} \label{equ:Non_Matrix rep_2}  
    t \begin{pmatrix}  
        1 \\ 0  
    \end{pmatrix} = R\left(\frac{1}{\sigma}\right)L_{2N-1}(k)R(\sigma)L_{2N-2}(k)R\left(\frac{1}{\sigma}\right)\cdots R\left(\frac{1}{\sigma}\right)L_1(k)R(\sigma)  
    \begin{pmatrix}  
        0 \\ 1  
    \end{pmatrix}, \quad t \neq 0,  
\end{align}  
where \(R(z)\) is defined as in (\ref{def:rlmatrix}), and  
\[
L_i(k) := \begin{pmatrix}  
\frac{2k\i}{\sqrt{\Delta_i}} \sinh\bigl(\frac{\sqrt{\Delta_i} t_i}{2}\bigr) + \cosh\bigl(\frac{\sqrt{\Delta_i} t_i}{2}\bigr) & \frac{\beta_i}{\sqrt{\Delta_i}} \sinh\bigl(\frac{\sqrt{\Delta_i} t_i}{2}\bigr) \\  
\frac{\beta_i}{\sqrt{\Delta_i}} \sinh\bigl(\frac{\sqrt{\Delta_i} t_i}{2}\bigr) & -\frac{2k\i}{\sqrt{\Delta_i}} \sinh\bigl(\frac{\sqrt{\Delta_i} t_i}{2}\bigr) + \cosh\bigl(\frac{\sqrt{\Delta_i} t_i}{2}\bigr)  
\end{pmatrix}.
\]

As in the previous sections, the equation \eqref{equ:Non_Matrix rep_2} (or its normalized version) forms the basis for deriving the resonances of problem \eqref{eq:coupled ods}. However, unlike the Hermitian case, the matrix \(L_i\) is not diagonal, which makes the characteristic function for resonant frequencies more intricate than \(f(k; \sigma)\) in Theorem \ref{thm: f(z;mu) zeros}. Consequently, we focus on subwavelength resonances and leave the analysis of the non-subwavelength regime for future work.

\subsection{Capacitance matrix theory} 
In this section, we develop the capacitance matrix theory for the subwavelength resonances of problem (\ref{eq:coupled ods}), building on the propagation matrix framework.

Similarly to \eqref{equ: M_total def}, define the modified total propagation matrix associated with \eqref{eq:coupled ods}: 
\[
M_{tot}(k; \sigma) := \frac{(4\sigma)^N}{(1+\sigma)^{2N}} R\left(\frac{1}{\sigma}\right)L_{2N-1}(k)R(\sigma)L_{2N-2}(k)\cdots R\left(\frac{1}{\sigma}\right)L_1(k)R(\sigma)\,,
\]  
where, using \eqref{def:matrixrs}, 
\begin{align*}
    \frac{4\sigma}{(1+\sigma)^2}P R\left(\frac{1}{\sigma}\right)L(k)R(\sigma)P = (R + \nu S)L(k)(R + \nu S), \quad \text{with} \quad P = \begin{pmatrix}  
    -1 & \\  
    & 1  
\end{pmatrix}.
\end{align*}
Thus, it follows that  
\[
P M_{tot}(k; \sigma) P = G(k; \nu),  
\]  
with 
\begin{equation} \label{equ:defGknunonhermitian1}  
G(k, \nu) := (R + \nu S)L_{2N-1}(k)(R + \nu S)L_{2N-2}(k)(R + \nu S)\cdots L_1(k)(R + \nu S).  
\end{equation}  
By (\ref{equ:Non_Matrix rep_2}), \(\omega = kv\) is a resonant frequency if and only if \(k\) is a zero of  
\[
g(k, \nu) := G(k, \nu)_{2,2} = g_0(k) + g_1(k)\nu + g_2(k)\nu^2 + \cdots + g_{2N}(k)\nu^{2N}.  
\]

Similarly to what we have done in Section \ref{sec:capacitancematrixtheroy1}, applying the Newton polygon method to the above expansion of $g(k, \nu)$, we can recover the capacitance matrix theory for the subwavelength resonances in \cite{ammari2024mathematical}. 
We define the generalized capacitance matrix as
\begin{align} \label{equ: Non_cap_matrix}
    \mathcal{C}:=\begin{pmatrix}
        \theta_1&-\theta_1\\
        -\theta_2&\theta_2+\theta_3&-\theta_3\\
        &-\theta_4&\theta_4+\theta_5&-\theta_5\\
        &&\ddots&\ddots&\ddots\\
        &&&-\theta_{2N-4}&\theta_{2N-4}+\theta_{2N-3}&-\theta_{2N-3}\\
        &&&&-\theta_{2N-2}&\theta_{2N-2}
    \end{pmatrix},
\end{align}
where 
\begin{align}
    \theta_j&:=\begin{cases}\label{equ:x_j}
        \frac{\beta_j \e^{\frac{\beta_j t_j}{2}}}{2\sinh(\frac{\beta_j t_j}{2})t_{j+1}}, &j=2i-1,\quad i=1,2,\cdots,N,\\
        \frac{\beta_{j+1} \e^{-\frac{\beta_{j+1}t_{j+1}}{2}}}{2\sinh(\frac{\beta_{j+1}t_{j+1}}{2})t_{j}}, &j=2i,\quad i=0, 1,2,\cdots,N-1,
    \end{cases}
\end{align}
where \(t_0=t_{2N}=1\).
\begin{remark}
    Noting that $\bm{t} = (r\ell_1, s_1, r\ell_2, \ldots, s_{N-1}, r\ell_N)$, and $\bm{\beta}$ is given by $\bm{\beta} = (\tfrac{\g_1}{r}, 0, \tfrac{\g_2}{r}, \ldots, 0, \tfrac{\g_N}{r})$. Substituting these into \eqref{equ: Non_cap_matrix} and \eqref{equ:x_j}, a direct computation yields the explicit form of $\mathcal{C}$, which is related to the gauge capacitance matrix $\mathcal{C}^{\gamma}$ in \cite[Corollary 2.6]{ammari2024mathematical} by $\mathcal{C} = V^{-1}\mathcal{C}^{\gamma}/r$.
\end{remark}

\begin{theorem} \label{thm:nonrecip}
    Let $\d\in\mathbb C$. For cases when $\d \to 0$ , the scattering problem \eqref{eq:coupled ods} has exactly $2N$ subwavelength resonant frequencies:
    \begin{itemize}
        \item[$\bullet$] a trivial frequency $\omega^*=0$.
        \item[$\bullet$] the first eigenfrequency $\omega_1(\d)$, which is an analytic function of $\d$ (or $\d^{-1}$) with leading asymptotic expansion:
        \begin{align*}
    &\ww_1(\d)=O(\d^2)+\\
    &\frac{-\i \delta v_b^2\left[\prod\limits_{j=1}^N\bigl(\frac{\g_{2j-1}}{2}\e^{\frac{\g_{2j-1} r\ell_{2j-1}}{2}}\bigr)+\prod\limits_{j=1}^N\bigl(\frac{\g_{2j-1}}{2}\e^{-\frac{\g_{2j-1}r\ell_{2j-1}}{2}}\bigr)\right]}
    {
    v\sum\limits_{i=1}^{N}\left[\sinh\bigl(\frac{\g_{2i-1}r\ell_{2j-1}}{2}\bigr)
    \prod\limits_{j=1}^{i-1}\bigl(\g_{2j-1}\e^{-\frac{\g_{2j-1}r\ell_{2j-1}}{2}}\bigr)
    \prod\limits_{j=i+1}^{N}\bigl(\g_{2j-1}\e^{\frac{\g_{2j-1}r\ell_{2j-1}}{2}}\bigr)\right]
    }. 
\end{align*}
        \item[$\bullet$] the remaining $2N-2$ frequencies are analytic functions of $\d^{1/2}$, with leading-order asymptotic expansions:
        \begin{align*}
        &\ww_i^{\pm}(\d)=\pm v\sqrt{\frac{\lad_i}{r}\d}-\i \d\frac{v_b}{2}\frac{\theta_{0}a_{i1}b_{i1}+\theta_{2N-1}a_{iN}b_{iN}}{\sum_{j=1}^{N}a_{ij}b_{ij}}+O(\d^{\frac{3}{2}}), \quad i=2,3, \dots, N, 
        \end{align*}
       where $0=\lambda_1<\lambda_2<\cdots<\lambda_N$ are the eigenvalues of $\mathcal{C}$, and $\bm{a}_i=(a_{ij})_{j=1}^N{}^\top$, $\bm{b}_i=(b_{ij})_{j=1}^N$ are the right and left eigenvectors corresponding to $\lambda_i$, respectively. 
    \end{itemize}
    What's more, the eigenmodes corresponding to the resonant frequencies $\omega_i=k_iv$ with $k_i=\sqrt{\lambda_i\d}+O(\d)$, $i=1,\ldots, N$, have the same form as \eqref{equ: u(x) estimate}, where $\bm a=(a_1,a_2,\ldots,a_N)^{\top}$ is the corresponding eigenvector of the capacitance matrix $\mathcal{C}$ and $b_j=\tfrac{a_{j+1}-a_j}{s_j}$ for $j=1,2,\ldots,N-1$. 
\end{theorem}

\section{Concluding remarks}
In this work, we have analyzed the scattering resonances of general one-dimensional acoustic media, identifying them as the zeros of an explicit trigonometric polynomial using the propagation matrix method. Leveraging Nevanlinna theory, we have established the global distribution properties of these resonances and characterized the resonance-free region by demonstrating the uniform boundedness of their imaginary parts. Additionally, we have derived asymptotic expansions for both subwavelength and non-subwavelength resonances in terms of the high contrast parameter. In the subwavelength regime, we have further employed the Newton polygon method to establish connections between our approach and the capacitance matrix theory for Minnaert resonances. This work lays the foundation for the analysis of scattering resonances in acoustic media with high-contrast bulk modulus, as well as the extension of capacitance matrix theory to non-subwavelength resonances. These topics will be investigated in forthcoming papers.


\appendix
\section{Asymptotic analysis of zeros via Newton polygon method}\label{app: Newton Polygon Method}
In this section, we review the Newton polygon method \cite{walker1978algebraic} for obtaining the asymptotics of the zeros of an analytic function.
We consider a function \(f(u, v)\) that is analytic in a region \(U \subset \C^2\), where \(u\) is treated as the variable and \(v\) as a parameter. Suppose that for \(v = v_0\),  the function \(f(\cdot, v_0)\) has an \(n\)-fold zero at \(u = u_0\):
\begin{equation*}
    \frac{\partial^k f}{\partial u^k}(u_0, v_0) = 0, \quad \text{for\ \, } k = 0, \dots, n-1, \quad \frac{\partial^n f}{\partial u^n}(u_0, v_0) \neq 0.
\end{equation*}
We next analyze the asymptotic behavior of the zeros \(u(v)\) of $f(\cdot,v)$ as \(v \to v_0\) using the Newton polygon method. For simplicity, we set \((u_0, v_0) = (0, 0)\). Expanding \(f\) at \((0, 0)\), we have:  
\begin{align}\label{equ: f(u,v) expand}  
f(u, v) = \sum_{j, k \geq 0} c_{jk} u^j v^k, \quad c_{j0} = 0 \ \text{for} \ j < n, \ \text{and} \ c_{n0} \neq 0.  
\end{align}
It suffices to find the asymptotic behavior of solutions $u(v)$ to $f(u(v),v)=0$ as $v\to0$. The Newton polygon is constructed as follows:
\begin{enumerate}
    \item For each nonzero coefficient $c_{jk}$ in (\ref{equ: f(u,v) expand}), plot $(j,k)$ in the plane $\R^2$.
    \item Form the convex hull of the set $\{(j,k):c_{jk}\ne0\}$.
    \item The lower boundary of this convex hull consists of a sequence of linear segments (edges), which determines the dominant asymptotic regimes.
\end{enumerate}
Each edge $l$ connecting points $(j_1,k_1)$ and $(j_2,k_2)$ with $j_1<j_2$ has the slope
\begin{equation} \label{eq:slope}
    \lambda_l=\frac{k_2-k_1}{j_2-j_1}.
\end{equation}
The dominant scaling exponent for $v$ corresponding to this edge is $s_l=-\lambda_l$, implying the asymptotic form $u\sim cv^{s_l}$ with $c\ne0$. The dominant balance equation for edge $l$ is derived as follows:
\begin{enumerate}
    \item Selecting all terms $(j,k)$ lying on the line $l$.
    \item Substituting the asymptotic form $u=cv^{s_l}$ in the equation
    \[
    \sum_{(j,k)\in l\cap\Z^2}c_{jk}u^jv^k=0
    \]
    to obtain the balance equation in $c$
    \begin{align}\label{equ: balance equation}
    \sum_{(j,k)\in l\cap\Z^2}c_{jk}c^j=0.
    \end{align}
    \item The left side of (\ref{equ: balance equation}) is a polynomial in $c$. The number of non-zero solutions $c$ (counting multiplicities) equals to the number of asymptotic branches associated with $l$. 
\end{enumerate}
For each asymptotic branch, the full expansion is given by a Puiseux series
\begin{align}\label{equ: u(v) expand}
u(v)=\sum_{m=0}^\infty d_mv^{s_l+t_m},
\end{align}
where $0=t_0<t_1<t_2<\cdots$ are rational exponents determined recursively. The coefficients \(d_m\) are determined as follows:  
\begin{enumerate}
    \item Substitute the series (\ref{equ: u(v) expand}) into \(f(u, v) = 0\) and expand it in powers of \(v\). 
    \item Match terms at each order of \(v\) and solve sequentially for each \(d_m\).
\end{enumerate} 
The exponents \(t_m\) are determined by the geometry of the Newton polygon above the initial edge. By Rouche's Theorem, the total number of asymptotic branches (counting multiplicities) must equal \(n\), consistent with the multiplicity of the zero \(u_0\) of \(f(\cdot, v_0)\).

\section{Proof of Proposition \ref{prop:ex_g}}\label{sect:proofexpangl}
\begin{proof}
Since $g(k;\nu):=G(k,\nu)_{2,2}$, it suffices to derive  the following expansions of $G_l(k)$ in (\ref{equ:G_expession}) as $k\rightarrow0$: for $1\leq l \leq N-1$,
\begin{equation} \label{expGl}
      \begin{aligned}
        G_l(k)=\bigl(\prod_{j=1}^{2N-1}t_j\bigr)&\left[ 2^l(-2\i k)^{2N-1-2l}\sum_{1\leq j_1\prec j_2\prec \cdots\prec j_l\leq 2N-2}\bigl(\prod_{i=1}^{l}\theta_{j_i}\bigr)R \right.\\
        &+\theta_02^{l-1}(-2\i k)^{2N-2l}\sum_{2\leq j_1\prec j_2\prec \cdots\prec j_{l-1}\leq 2N-2}\bigl(\prod_{i=1}^{l-1}\theta_{j_i}\bigr)R_{-}\\
        & +\theta_{2N-1}2^{l-1}(-2\i k)^{2N-2l}\sum_{1\leq j_1\prec j_2\prec \cdots\prec j_{l-1}\leq 2N-3}\bigl(\prod_{i=1}^{l-1}\theta_{j_i}\bigr)R_{+}\\
        &\left.+O(k^{2N-2l+1}) \right],
    \end{aligned}
\end{equation}
    and for $l=0, N$,
    \begin{equation} \label{expGl2}      G_0(k)=\bigl(\prod_{j=1}^{2N-1}t_j\bigr)(-2k\i)^{2N-1}R+O(k^{2N+1}), \quad G_N(k)=2^{N-1}(R_{+}+R_{-})+O(k).
    \end{equation}
This expansion is based on (\ref{equ: Matrix rep4}) and the following identities:
\begin{equation}\label{equ:matrixexpan2}
\begin{aligned}
        & 1. \ RMR=\eta(M)R,\\
        & 2. \ \eta(L_j)=-2t_jk\i+O(k^3),\q \eta(L_{j+1}SL_{j})=2+O(k^2),\\
        & 3.\ \tau_0\left(\eta(L_j)\right)=1,\q \tau_0(\eta(L_{j+1}SL_{j}))=0,\q \tau_0(\eta(L_{j+2}SL_{j+1}SL_{j}))\geq 1,\\
        & 4.\ SL_{2N-1}(k)R=R_-+O(k), \q RL_{1}(k)S=R_++O(k), \q\\
    \end{aligned}
\end{equation}
where $\eta(M) = M_{21}+M_{22}-M_{11}-M_{12}$ for $M\in \mathbb{C}^{2\times 2}$, and $\tau_0(f)$ denotes the order of the leading-order term in the asymptotic expansion of $f(k)$ as $k\to 0$, in particular, $\tau_0(0)=+\infty$.

Firstly, for $1\leq l\leq N-1$, by (\ref{equ: Matrix rep4}), the expression of $G_l(k)$ consists of $\tbinom{2N}{l}$ terms, each corresponding to a selection of $l$ $S$-matrices (or equivalently, $2N-l$ $R$-matrices) from the $2N$ available matrices $R + \nu S$.  We denote by $Q$ the matrix product corresponding to a specific selection, and we will analyze its order in $k$ as $k\to 0$. 

To this end, we first consider the $2N-l-1$ gaps formed by the $2N-l$ $R$-matrices and denote the product of all matrices within each gap by $M_1^{a_1}, M_2^{a_2},\ldots, M_{2N-l-1}^{a_{2N-l-1}}$, where $a_i$ denotes the number of $S$-matrices in the $i$-th gap; see Fig.\ref{fig: M_s_basic} for an illustration. Moreover, we denote by $M_L^{a_L}$ the product of all matrices to the left of the first $R$-matrix, and by $M_R^{a_R}$ the product of all matrices to the right of the last $R$-matrix, where $a_L$ and $a_R$ denote the corresponding numbers of $S$-matrices involved. In particular, $M_L^0$ and $M_R^0$ are identity matrices.

\begin{figure}[htb]
    \centering
    \begin{tikzpicture}[
        block/.style={rectangle, draw, inner sep=3pt, font=\footnotesize},
        label/.style={font=\footnotesize, below=2pt}
    ]
        \node[block] (ML1) at (0,0) {$SL_{2N-1}$};
        \node[label] at (ML1.south) {$M_L^1$};
        \node[font=\footnotesize, right=2pt of ML1] (R1) {$R$};
        \node[block, right=2pt of R1] (M10) {$L_{2N-2}$};
        \node[label] at (M10.south) {$M_1^0$};
        \node[font=\footnotesize, right=2pt of M10] (R2) {$R$};
        \node[block, right=2pt of R2] (M22) {$L_{2N-3}SL_{2N-4}SL_{2N-5}$}; 
        \node[label] at (M22.south) {$M_2^2$};
        \node[font=\footnotesize, right=2pt of M22] (dots) {$\cdots$};
        \node[font=\footnotesize, right=2pt of dots] (Rmid) {$R$};
        \node[block, right=2pt of Rmid] (Mend) {$L_4SL_3$};
        \node[label] at (Mend.south) {$M_{2N-1-l}^1$};
        \node[font=\footnotesize, right=2pt of Mend] (Rlast) {$R$};
        \node[block, right=2pt of Rlast] (MR2) {$L_2SL_1S$}; 
        \node[label] at (MR2.south) {$M_R^2$};
    \end{tikzpicture}
    \caption{Structure of $Q$ with $a_L=1,a_1=0,a_2=2,\cdots, a_{2N-l-1}=1,a_R=2$.}
    \label{fig: M_s_basic}
\end{figure}

Then, by \eqref{equ:matrixexpan2}(1), the total matrix product $Q$ can be written as 
\begin{equation}\label{equ:msexpress1}
Q=\left[\prod_{i=1}^{2N-l-1}\eta(M_i^{a_i})\right] M_L^{a_L}RM_R^{a_R}.    
\end{equation}
Since we select exactly $l$ $S$-matrices, we have $a_L + a_R + \sum_{i=1}^{2N-l-1}a_i=l$, which implies that at least $2N-2l-1$ of the $a_i$'s must be zero. Furthermore, applying \eqref{equ:matrixexpan2}(3) to (\ref{equ:msexpress1}) and using the definition of $\tau_0$, we obtain
\begin{align}\label{equ:tau_0(Q)}
\tau_0(Q)\geq\sum_{i=1}^{2N-l-1}\tau_0(\eta(M_i^{a_i}))\geq\sum_{i=1}^{2N-l-1}\tau_0(\eta(M_i^{a_i}))\mathbf{1}_{a_i=0}\geq 2N-2l-1.
\end{align}
The equality holds if and only if we have $2N-2l-1$ $a_i$'s equal to zero, the remaining $l$ $a_i$'s are one, and $a_L=a_R=0$. In this case (i.e., $\tau_0(Q) = 2N - 2l - 1$), the $l$ matrices $M_{i}^{a_i}$ with $a_i=1$ have the form $L_{j_i+1}SL_{j_i}$ with $1\leq j_1\prec j_2\prec \cdots\prec j_l\leq 2N-2$. Then, according to \eqref{equ:matrixexpan2}(2), the leading-order term in the expansion of $Q$ is given by 
\begin{align}\label{equ: eta=2N-2l-1}
    Q=&\left[\prod_{i=1}^{2N-l-1}\eta(M_i^{a_i})\right] R=\prod_{j\neq j_i ,j_{i+1}}\eta(L_j)\prod_{i=1}^{l}\eta(L_{j_i +1}SL_{j_i})R\\
    =&\left[ 2^l(-2\i k)^{2N-1-2l}\prod_{j=1}^{2N-1}t_j\prod_{i=1}^{l}\theta_{j_i} \right]R+O(k^{2N-2l+1})\notag.
\end{align}

Now consider the selections that yield $\tau_0(Q)=2N-2l$. In this case, the equality condition in \eqref{equ:tau_0(Q)} implies that exactly $2N-2l$ of the $a_i$'s equal to zero. Since $a_L+a_R+\sum_{i=1}^{2N-l-1}a_i\mathbf{1}_{a_i\neq 0}=l$, it suffices to consider the following two cases: 

\noindent $\bullet$ $a_L=a_R=0$, exactly one $a_i$ equals to two, and all remaining non-zero $a_i$'s equal to one. Then, 
\begin{align*}
    \tau_0(Q)=\sum_{i=1}^{2N-l-1}\tau_0(\eta(M_i^{a_i}))=\sum_{i=1}^{2N-l-1}\tau_0(\eta(M_i^{a_i}))(\mathbf{1}_{a_i=0}+\mathbf{1}_{a_i=2})\geq 2N-2l+1. 
\end{align*}
In this case, $\tau_0(Q)$ cannot be $2N-2l$.

\noindent $\bullet$ $a_L=1, \ a_R=0$ or $a_R=1, \ a_L=0$, and all nonzero $a_i$ equal to one. We first expand $Q$ for the case $a_L=1, \ a_R=0$. In this case, the $l-1$ matrices $M_{i}^{a_i}$ with $a_i=1$ have the form $L_{j_i+1}SL_{j_i}$, where $1\leq j_1\prec j_2\prec \cdots\prec j_{l-1}\leq 2N-3$. Then, using \eqref{equ:matrixexpan2}(2)(4), we obtain 
\begin{align}\label{equ: eta=2N-2l}
    Q=2^{l-1}(-2\i k)^{2N-2l}\theta_{2N-1}\prod_{j=1}^{2N-1}t_j\prod_{i=1}^{l-1}\theta_{j_i}R_{+}+O(k^{2N-2l+1}).
\end{align}
Similarly, we can expand $Q$ for the case when $a_R=1, \ a_L=0$ that
\[
 Q=2^{l-1}(-2\i k)^{2N-2l}\theta_{0}\prod_{j=1}^{2N-1}t_j\prod_{i=1}^{l-1}\theta_{j_i}R_{-}+O(k^{2N-2l+1}).
\]
For the remaining selections, we have $\tau_0(Q) \geq 2N - 2l + 1$, and they are absorbed into the $O(k^{2N - 2l + 1})$ term in \eqref{expGl}. Having established the expansion of $G_l(k)$ for $1\leq l \leq N-1$, the proof is completed by deriving the expressions at the boundary indices $l=0$ and $l=N$.

For \(l = 0\), there is only one possible selection (i.e., selecting all \(R\)-matrices from the \(2N\) available matrices \(R + \nu S\)), which gives  
\[
G_0(k) = \left[\prod_{i=1}^{2N-l-1} \eta(L_i)\right] R = \Bigl(\prod_{j=1}^{2N-1} t_j\Bigr)(-2k\i)^{2N-1}R + O(k^{2N+1}).
\]

For \(l = N\), for any selection and the corresponding product \(Q\), it follows from the inequality \eqref{equ:tau_0(Q)} that when \(\tau_0(Q) = 0\), all \(a_i\)'s are equal to one, and either \(a_L\) or \(a_R\) is equal to one. By absorbing the product of other selections into the \(O(k)\) term, we obtain  $G_N(k) = 2^{N-1}(R_{+} + R_{-}) + O(k)$ as desired. 

\end{proof}


\bibliographystyle{siamplain}
\bibliography{highcon}
\end{document}